\documentclass[11pt,reqno,a4paper,english]{amsart}

\usepackage[utf8]{inputenc}  
\usepackage{dsfont}
\usepackage{xcolor}

\usepackage{enumerate}          
\usepackage{amssymb,amsmath,amsfonts,amsthm}            
\usepackage{pdfsync}
\usepackage{url}
\usepackage{hyperref}
\usepackage{anysize}
\usepackage{mathrsfs}
\usepackage{graphicx,float,color}
\usepackage{epsfig}
\usepackage{tikz}

\begin{document}
\theoremstyle{plain}
\newtheorem{theorem}{Theorem}[section]
\newtheorem{definition}[theorem]{Definition}
\newtheorem{proposition}[theorem]{Proposition}
\newtheorem{lemma}[theorem]{Lemma}
\newtheorem{corollary}[theorem]{Corollary}
\newtheorem{conjecture}[theorem]{Conjecture}
\newtheorem{remark}[theorem]{Remark}
\numberwithin{equation}{section}
\errorcontextlines=0

\renewcommand{\P}{\mathcal{P}}
\newcommand{\Id}{\text{Id}}
\newcommand{\R}{\mathbb{R}}
\newcommand{\M}{\mathbf{H}^m}
\newcommand{\Op}{{\rm Op}}
\newcommand{\J}{\mathcal{J}}
\newcommand{\Z}{\mathbb{Z}}
\newcommand{\N}{\mathbb{N}}
\newcommand{\supp}{\text{supp}}
\newcommand{\ord}{\text{ord}}
\newcommand{\e}{\varepsilon}
\newcommand{\bqn}{\begin{equation}}
\newcommand{\eqn}{\end{equation}}
\newcommand{\al}{\alpha}
\newcommand{\tor}{(2\mathbb{T})^{2}}
\newcommand{\er}{{\vec{e_{r}}}}
\newcommand{\et}{{\vec{e_{\theta}}}}
\newcommand{\red}{\color{red}}
\newcommand{\black}{\color{black}}
\renewcommand{\i}{\iota}
\newcommand{\todo}[1]{$\clubsuit$ {\tt #1}}
\newcommand{\B}{\mathcal{B}}
\newcommand{\BST}{\textbf{(BST) }}
\newcommand{\BSCT}{\textbf{(BSCT)}}
\newcommand{\Green}{\textbf{(Green)}}
\newcommand{\minmult}{\min_n}
\newcommand{\dproj}{\widetilde{\Pi}_I}
\newcommand{\TV}{\rm TV}

\setcounter{tocdepth}{1}

\title[Delocalized eigenvectors of transitive graphs]{Delocalized eigenvectors of transitive graphs\\ and beyond}
%
\author{Nicolas Burq}
\address{Nicolas Burq. Laboratoire de Math\'ematiques d'Orsay, Universit\'e
  Paris-Sud, Universit\'e Paris-Saclay, B\^atiment~307, 91405
  Orsay Cedex \& CNRS UMR 8628 \& Institut Universitaire de France}
 \email{nicolas.burq@universite-paris-saclay.fr}

\author{Cyril Letrouit}
\address{Cyril Letrouit. Laboratoire de Math\'ematiques d'Orsay, Universit\'e
  Paris-Sud, Universit\'e Paris-Saclay, B\^atiment~307, 91405
  Orsay Cedex \& CNRS UMR 8628}
 \email{cyril.letrouit@universite-paris-saclay.fr}

\date{\today}
\maketitle

\begin{abstract}
We prove delocalization properties for eigenvectors of vertex-transitive graphs via elementary estimates of the spectral projector. We recover in this way known results which were formerly proved using representation theory.\\ Similar techniques show that for general symmetric matrices, most approximate eigenvectors spectrally localized in a given window containing sufficiently many eigenvalues are delocalized in $L^q$ norms. Building upon this observation, we prove a delocalization result for approximate eigenvectors of large graphs containing few short loops, under an assumption on the resolvent which is verified in some standard cases, for instance random lifts of a fixed base graph.
\end{abstract}

\tableofcontents

\section{Introduction}
\subsection{Overview} 


Let $A$ be the adjacency matrix of a graph with vertex set $[n]=\{1,\ldots,n\}$. The subject of this paper is the spatial delocalization of the eigenvectors and approximate eigenvectors of $A$ in the limit of large $n$. To measure the delocalization of a vector $u\in\mathbb{C}^n$, we use the commonly considered quantities
$$
\alpha_q(u)=\frac{\|u\|_{L^q}}{\|u\|_{L^2}}
$$
for $q\in(2,+\infty]$ (see \cite{deluca}, \cite{fyodorov}, \cite{tikhonov}, \cite{bordenave}). Informally, a sequence $(u_n)_{n\in\N}$, $u_n\in\mathbb{C}^n$, is localized in the $L^q$ sense as $n\rightarrow +\infty$ if $\alpha_q(u_n) \asymp 1$ and completely delocalized if $\alpha_q(u_n)=n^{\frac1q -\frac12+o(1)}$.




The literature on localization/delocalization of eigenvectors of graphs and matrices is huge. Our purpose in this paper is to show that estimates on spectral projectors of $A$ coupled with concentration of measure arguments allow to prove delocalization for:
\begin{enumerate}
\item \textit{Eigenvectors} of adjacency matrices of \textit{vertex-transitive graphs} (e.g. Cayley graphs) and generalizations thereof.  
\item \textit{Most approximate eigenvectors} of general symmetric matrices; here approximate means that instead of satisfying $A_nu_n=\lambda_n u_n$, approximate eigenvectors satisfy $A_nu_n=\lambda_n u_n+o_{L^2}(\lambda_n u_n)$. In this context, delocalization holds with high probability when the approximate eigenvectors are chosen randomly in a fixed spectral window.
\end{enumerate}
Concerning Point 1, we recover known results (\cite{mageethomas}, \cite{zhao}) in a more direct way, avoiding in particular any use of representation theory. Concerning Point 2, our arguments show for instance that for adjacency matrices of large graphs satisfying two assumptions (on the number of small cycles and on the Green function of their local weak limits), linear combinations of $o(n)$ eigenvectors are completely delocalized with high probability.

We underline that our results on delocalization of \textit{(exact) eigenvectors} are restricted to vertex-transitive graphs. Very strong tools have been developed in the past 15 years for other classes of graphs and matrices, e.g., regular graphs, Erdös-Rényi graphs, expander graphs with few cycles, Wigner matrices and generalizations thereof, Lévy matrices, etc (see Section \ref{s:biblio}). Our results of Point 2 are reminiscent of \cite{vanhandel}, in which the authors exhibit a similar phenomenon of delocalization of \textit{approximate eigenvectors} in sparse Gaussian random matrices with quite general sparsity patterns.

\subsection{Delocalized random eigenbases of vertex-transitive graphs}
Our first results are concerned with vertex-transitive graphs, i.e., graphs for which the automorphism group\footnote{An automorphism of a graph $G = (V, E)$ is a permutation $\sigma$ of the vertex set $V$, such that if $u,v\in V$, then $(u, v)\in E$ if and only $(\sigma(u),\sigma(v))\in E$.} acts transitively on their vertices (e.g., Cayley graphs).

Given a graph $G$, we denote by $\mathcal{B}(G)$ the set of orthonormal bases of the adjacency matrix of $G$. We explain in Section \ref{s:symmetricgraphs} how to define a uniform probability measure $\nu$ on $\mathcal{B}(G)$.
\begin{theorem}[$L^q$-delocalized eigenbases]\label{t:trans}
There exists $C>0$ such that the following holds. Let $n\in\mathbb{N}^*$ and let $G$ be a vertex-transitive graph with $n$ vertices. Then for any $\Lambda>0$, with probability $\geq 1-n^{2-\log(\Lambda)}$ on the choice of an element $\mathcal{B}\in \mathcal{B}(G)$ picked following $\nu$, any $u\in\mathcal{B}$ verifies
\begin{equation}\label{e:Lqbasis}
\|u\|_{L^\infty}\leq C\Lambda \left(\frac{\log n}{n}\right)^{\frac12}.
\end{equation} 
Also, there exists $C>0$ such that for any $q\in[2,+\infty)$ and any $\Lambda>0$, with probability $\geq 1-n\Lambda^{-q}$ on the choice of an element  $\mathcal{B}\in \mathcal{B}(G)$ picked following $\nu$, any $u\in\mathcal{B}$ verifies
\begin{equation}\label{e:Lqbasisq}
\|u\|_{L^q}\leq C\Lambda\sqrt{q} n^{\frac1q-\frac12}.
\end{equation}
\end{theorem}

The first part \eqref{e:Lqbasis} has already been proved in \cite[Theorem 1.4]{zhao} but we give a much shorter proof avoiding representation theory. We also provide in Theorem \ref{t:product} an extension of Theorem \ref{t:trans} to the case of products of two graphs, when one of the two graphs is vertex-transitive. 

It is proved in \cite[Theorem 1.2]{zhao} that there exist infinitely many Cayley graphs $G$ whose adjacency matrix has an eigenspace all of whose eigenvectors $u:G\rightarrow \mathbb{C}$ satisfy $\|u\|_{L^\infty} \geq \|u\|_{L^2}(\log \log n)^{-1}\sqrt{\log n / n}$, where $n$ is the number of vertices and $c > 0$ is some absolute constant. This shows that Theorem \ref{t:trans} is almost sharp: for instance, we cannot replace the right-hand side in \eqref{e:Lqbasis} by $\|u\|_{L^\infty}\lesssim n^{-1/2}$, which would be the best bound we could hope for.

Our second result shows that  in sufficiently large eigenspaces of vertex-transitive graphs, the statistics of the entries of random eigenvectors are Gaussian. Our result is stated in terms of the bounded Lipschitz distance  $d_{\text{BL}}$ between probability measures $\mu,\nu$ on $\R$, defined as
$$
d_{\rm BL}(\mu,\nu)=\sup_{\|f\|_{\rm BL}\leq 1} \left|\int_{\R} fd\nu - \int_{\R} fd\mu\right|
$$
where
$$
\|f\|_{\rm BL}=\max\left\{\|f\|_{L^\infty},\ \sup_{x\neq y}\frac{|f(x)-f(y)|}{|x-y|}\right\}.
$$
This distance metrizes the weak convergence of probability measures.
\begin{theorem}[Gaussian statistics]\label{t:Gaussianstat}
Let $G$ be a vertex-transitive graph whose set $V$ of vertices has $n$ elements, and let $E$ be an eigenspace of the adjacency matrix of $G$ with dimension $m$. Let $u\in\R^n$ be an eigenvector chosen according to the uniform probability measure $\mathbb{P}$ on the unit sphere of $E$. Set
$$
\mu=\frac1n \sum_{i\in V} \delta_{\sqrt{n}u_i}.
$$
Then for any $\varepsilon$ such that $\max\left(\left(\frac{192 \pi}{\sqrt{m-1}}\right)^{2/5},\frac{4}{m-1}\right)\leq \varepsilon\leq 1$, 
$$
\mathbb{P}\left[d_{BL}(\mu,\mathcal{N}(0,1))>\varepsilon\right]\leq \frac{48\sqrt{\pi}}{\varepsilon^{3/2}} \exp\left(-c(m-1)\varepsilon^5\right)
$$ 
where $c=3^{-2}2^{-16}$ and $\mathcal{N}(0,1)$ denotes the standard Gaussian.
\end{theorem}
Notice that if $m$ is small, the statistics of the entries of random eigenvectors are not necessarily Gaussian: a good example is provided by cycle graphs where multiplicities are equal to $1$ or $2$, and statistics are not Gaussian. Results similar to Theorem \ref{t:Gaussianstat} exist for other models, see for instance \cite[Theorem 2.4]{shou}.


Our third result has already been proved in \cite[Theorems 1.1 and 1.8]{mageethomas} for Cayley graphs, which are particular vertex-transitive graphs, but again we provide a much shorter proof avoiding representation theory. This result says that typical eigenbases of vertex-transitive graphs are delocalized (in a ``quantum ergodic" sense) when most multiplicities are large.  
\begin{theorem}\label{t:smallscale}
Let $G$ be a vertex-transitive connected graph, with vertex set $V$. We denote by $m_k$, $1\leq k\leq K$, the multiplicities of the distinct non-trivial eigenvalues of the adjacency matrix of $G$. Let $M\in\N$ and let $f_1,\ldots,f_M\in L^2(V)$ be a collection of real-valued functions. 

Then, for any $t>0$, with probability at least
\begin{equation}\label{e:mult}
1-M\sum_{k=1}^K m_k\left(3e^{-\frac{t\sqrt{m_k}}{8}}+e^{-\frac{m_k}{12}}\right)
\end{equation}
with respect to the choice of an orthonormal basis $\mathcal{B}\in\mathcal{B}(G)$ according to $\nu$, the following property holds: for any $u\in\mathcal{B}$ and any $i=1,\ldots,M$, 
$$
\left|\sum_{x\in V}f_i(x)u(x)^2-\frac{1}{|V|}\sum_{x\in V}f_i(x)\right|\leq t\frac{\|f_i\|_{L^2}}{\sqrt{|V|}}.
$$
\end{theorem} 
The multiplicities appearing in the statements of Theorems \ref{t:Gaussianstat} and \ref{t:smallscale} are often large; for instance they are large for Cayley graphs built on quasi-random groups, as emphasized in the introductions of \cite{mageethomas} and \cite{naor}. Recall that a group $H$ is called  $D$-quasi-random if all its non-trivial unitary representations have dimension at least $D$; this condition implies that the Cayley graphs built on these groups have non-trivial eigenvalues multiplicities $\geq D$. 

Finite simple groups of Lie type with rank $\leq r$ are $|H|^s$-quasi-random for some $s>0$ depending only on $r$. Hence, in this case, $m_k\geq n^s$ for some $s>0$ and any $k$ for which the eigenvalue is non-trivial. As shown in \cite[Corollary 1.6]{mageethomas}, Theorem \ref{t:smallscale}, applied to any Cayley graph built with a symmetric set of generators of $H$, says that if $A_i\subset H$ is a collection of subsets partitioning $H$ with sizes satisfying $c|H|^{1-\eta}\leq |A_i|\leq C|H|^{1-\eta}$ where $0<\eta<s$, then with high $\nu$-probability on the choice of an orthonormal eigenbasis $\mathcal{B}$ of the adjacency operator, for every $i$ and every $u\in\mathcal{B}$, 
$$
\Bigl|\sum_{x\in A_i}u(x)^2-\frac{|A_i|}{|H|}\Bigr|\leq \frac{K\log |H|}{|H|^{\frac12(s-\eta)}}\frac{|A_i|}{|H|}.
$$
where $K>0$ only depends upon $c$. 

The conclusion of Theorem \ref{t:smallscale} is reminiscent of the ``probabilistic quantum unique ergodicity" statement of \cite[Corollary 1.3]{bauer}, which is a result about the simultaneous delocalization of all bulk eigenvectors of random regular graphs (without any averaging over eigenvectors). In  \cite[Corollary 1.3]{bauer} as in Theorem \ref{t:smallscale}, delocalization is only tested against a certain number of observables, and this number appears in the bound. In our Theorem \ref{t:smallscale}, since we assume high multiplicities through assumption \eqref{e:mult}, our result holds only for most eigenbases: delocalization cannot hold for all eigenvectors if multiplicities become too large.

The interested reader will find many other applications of Theorem \ref{t:smallscale} in \cite{mageethomas}.

\begin{remark}
Let $G=\{g_1,\ldots,g_n\}$ be a group with cardinality $n$, and $\alpha: G\rightarrow \R$, and consider the $n\times n$ matrix whose $(i,j)$ coefficient is $\alpha(g_jg_i^{-1})$. This is the adjacency matrix of a complete graph with edges weighted by $\alpha$. If $\alpha$ is the characteristic function of a symmetric set of generators of $G$, we recover adjacency matrices of usual Cayley graphs. Theorems \ref{t:trans} and \ref{t:smallscale} also work for matrices associated to general functions $\alpha:G\rightarrow \R$.
\end{remark}

\begin{remark}
Theorems \ref{t:trans} and \ref{t:smallscale} may be adapted to the setting of \emph{quantum} Cayley graphs, with all edges of same length, in fixed spectral windows (but not over the whole spectrum, which is infinite for quantum graphs), when the number of vertices tends to $+\infty$. One should first prove equidistribution of the quantities $\|u\|_{L^q(e)}$ where $e$ runs over the edges of the graph, and then invoke the fact that on each edge, eigenfunctions are well-spread due to their simple form. 
\end{remark}


\subsection{Delocalized approximate eigenvectors of general symmetric matrices}\label{s:delocmatrix}
It turns out that the technique used to prove Theorem \ref{t:trans} also provides information for delocalization properties of approximate eigenvectors of general symmetric matrices. But these results, instead of working for (exact) eigenvectors, only give information for \emph{most approximate eigenvectors}; in particular there is no way to deduce from these results any precise information on \emph{true eigenvectors}, except in sufficiently degenerate eigenspaces.

This section is devoted to giving a precise statement for this basic but very useful observation. Given a symmetric $n\times n$ matrix $H$ and a subset $I\subset \R$ (both $H$ and $I$ depend on $n$, but we omit dependence in $n$ in the notation of this section), we denote by $N(I)$ the number of eigenvalues of $H$ in $I$. We consider an orthonormal basis $(\psi_{\lambda_k})_{k\in[n]}$ of eigenvectors (in $\R^n$) of $H$ with associated eigenvalues $\lambda_k$. We set\footnote{The subscript $\lambda_k\in I$ in the sum means that we are summing over all $k\in[n]$ such that $\lambda_k\in I$.}
\begin{equation}\label{e:EI}
E_I=\Bigl\{u=\sum_{\lambda_k\in I} z_k\psi_{\lambda_k}, z_k\in\R\Bigr\},
\end{equation}
which is isometric via $u\mapsto (z_k)$ to $\R^{N(I)}$ endowed with the scalar product coming from $\R^n$. We denote by $\mathbb{S}_I$ the unit sphere of $E_I$, and by $\mathbb{P}_I$ the uniform probability on $\mathbb{S}_I$. We pick
\begin{equation}\label{e:randomu}
u\in \mathbb{S}_I \text{  according to  } \mathbb{P}_I.
\end{equation}
If $\sup(I)-\inf(I)$ is not too large, then elements of $E_I$ are approximate eigenvectors. Indeed, if $\lambda\in I$ and $u=\sum_{\lambda_k\in I}z_k\psi_{\lambda_k}\in E_I$ we have 
$$
\|(H-\lambda)u\|_{L^2}=\Bigl\|\sum_{\lambda_k\in I}(\lambda_k-\lambda)z_k\psi_{\lambda_k}\Bigr\|_{L^2}\leq (\sup(I)-\inf(I))\|u\|_{L^2}.
$$
As soon as $I$ is a small interval (for instance $\sup(I)-\inf(I)\ll \lambda$), $Hu=\lambda u+o_{L^2}(\lambda u)$ as $n\rightarrow +\infty$.

The next result is relevant when $N(I)\rightarrow +\infty$ as $n\rightarrow +\infty$. It shows that in this case, $u$ drawn according to \eqref{e:randomu} is delocalized  in the sense of $L^q$ norms ($q\in (2,+\infty]$) with high $\mathbb{P}_I$-probability.  The larger $N(I)$ is, the better the estimates are. 
\begin{theorem}\label{t:delocgeneralsituation} There exists $C>0$ universal (not depending on $I$) such that if $u\in\mathbb{S}_I$ is a random vector with law $\mathbb{P}_I$, then:
\begin{enumerate}[(i)]
\item For any $q\in[2,+\infty)$ and any $\Lambda\geq 1$
$$
\mathbb{P}_I\left(\|u\|_{L^q}\geq C\Lambda\sqrt{q} N(I)^{\frac1q-\frac12}\right) \leq  4\exp\left(-\frac18 C^2\Lambda^2 qN(I)^{\frac2q}\right).
$$
\item For any $\Lambda\geq 1$
$$
\mathbb{P}_I\left(\|u\|_{L^\infty}\geq C'\Lambda \left(\frac{\log(N(I))}{N(I)}\right)^{\frac12}\right) \leq  4N(I)^{-\frac18 C'^2\Lambda^2}
$$ 
where $C'=Ce$.
\end{enumerate}
\end{theorem}
Theorem \ref{t:delocgeneralsituation} follows from elementary concentration of measure estimates\footnote{Theorem \ref{t:delocgeneralsituation} does not actually rely on the fact that $u$ is a linear combination of eigenvectors, but only on the fact that $u$ is a random linear combination of elements of an orthonormal basis $(\psi_{\lambda_k})_{k\in[n]}$ of $\R^n$. However, as explained above, in the framework of random linear combinations of eigenvectors, $u$ is automatically an approximate eigenvector - which is a nice property.}. The idea of using linear combinations to obtain delocalized approximate eigenvectors is not new (see Section \ref{s:biblio} for references), nevertheless, to the best of our knowledge, it has never been stated in the general and sharp form of Theorem \ref{t:delocgeneralsituation}. This result provides motivation for Section \ref{s:strongerdelocappro}, which strengthens the above bounds under some assumptions on $H$.

\begin{remark}
Results similar to Theorem \ref{t:delocgeneralsituation} also hold for linear combinations of eigenvectors of \emph{normal matrices}. In this case $I$ is a region of $\mathbb{C}$, eigenvectors are complex-valued, and $\mathbb{S}_I$ is the unit sphere of $\mathbb{C}^{N(I)}$. 
\end{remark}

\subsection{Stronger delocalization of approximate eigenvectors under two assumptions} \label{s:strongerdelocappro}
For fixed $q\in[2,+\infty)$, to obtain the optimal delocalization
\begin{equation}\label{e:optimal}
\|u\|_{L^q}\lesssim Cn^{\frac1q -\frac12}
\end{equation}
 with high probability on $u\sim \mathbb{P}_I$, Theorem \ref{t:delocgeneralsituation} requires $N(I)\gtrsim n$.

In this section, we push further the ideas of Section \ref{s:delocmatrix} and state two assumptions on families of large (deterministic) graphs for which it is possible to take smaller $I$, with $N(I)\lesssim n\frac{\log \log n}{\log n}$, while keeping the optimal delocalization \eqref{e:optimal} of approximate eigenvectors  with high probability.

These two assumptions have already been considered several times in the literature: the first assumption concerns the number of small cycles in the graphs, which is assumed to be small; under this assumption, the graphs converge locally weakly (i.e., in the sense of Benjamini-Schramm) toward a probability measure on rooted trees. The second assumption is a bound on the expectation of the Green function of the limiting rooted trees under this probability measure. 

These assumptions are close to those of the paper \cite{anansabri} where, under an additional assumption of expansion of the graphs which we do not need here, Anantharaman and Sabri prove a quantum ergodicity result. Their result is not strong enough to give information about $L^q$ norms of eigenvectors; our result gives information about $L^q$ norms, but only for most approximate eigenvectors, which again is much easier to obtain than for exact eigenvectors.

\subsubsection{Assumption of few short loops} \label{s:feshortloops} We consider a sequence of graphs $(G_n)_{n\in\N}$ with vertex set $V_n$, and degree bounded by $D$. Our first assumption says that the graphs have few short loops:

\BST For all $r>0$, $$\lim_{n\rightarrow \infty} \frac{|\{x\in V_n:\rho_{G_n}(x)<r\}|}{|V_n|}=0$$ where $\rho_{G_n}(x)$ is the injectivity radius of $G_n$ at $x$, i.e., the maximal radius $\rho$ for which the ball $B_{G_n}(x,\rho)$ is a tree.

Up to passing to a subsequence (which we omit in the notation), assumption \BST is equivalent to

\BSCT\ The sequence $G_n$ has a ``local weak limit" $\overline{\mathbb{P}}$ supported on the set of (isomorphism classes of) rooted trees.

We refer to Appendix \ref{a:localweak} for reminders on local weak limits, which are also called Benjamini-Schramm limits. Here we simply recall that \BSCT\ means that for any $h\in\N$ and any rooted graph $(H;o)$ with depth $h$, there holds
$$
\lim_{n\rightarrow \infty} \frac{|\{x\in V_n :(G_n;x)_h\simeq (H;o)\}|}{|V_n|}=\overline{\mathbb{P}}(\{(G,x) : (G;x)_h \simeq (H;o)\})
$$ 
where $(G_n;x)_h$ denotes the graph obtained by cutting $G_n$ at distance $h$ from $x$, and $\simeq$ is the symbol of graph isomorphy.

Let us reformulate \BSCT. We introduce for any $n\in\mathbb{N}$ and $h\in \mathbb{N}$ the probability measure on the finite set of (isomorphism classes of) rooted graphs with depth $\leq h$ given by\footnote{There is a slight abuse of notation here. The Dirac masses should actually be put on the isomorphism classes of $(G_n;x)_h$.}
\begin{equation}\label{e:PGnh}
\overline{\mathbb{P}}^{(h)}_{G_n}=\frac{1}{|V_n|} \sum_{x\in V_n} \delta_{(G_n;x)_h}.
\end{equation}
\
 Similarly for any  rooted graph $(H;o)$ with depth $\leq h$, we set
$$
\overline{\mathbb{P}}^{(h)}(H;o)=\overline{\mathbb{P}}(\{(G;x): (G;x)_h \simeq (H;o)\})
$$
(note that $\overline{\mathbb{P}}^{(h)}$ is supported on rooted \textit{trees}).
Then \BSCT\ is equivalent to 
\begin{equation}\label{e:dtvconv}
\forall h\in \N, \qquad d_{\TV}(\overline{\mathbb{P}}^{(h)}_{G_n},\overline{\mathbb{P}}^{(h)})\underset{n\rightarrow +\infty}{\longrightarrow} 0
\end{equation}
where $d_{\TV}$ denotes the total variation distance. 

 In Theorem \ref{t:Linftygreen}, concerning $L^q$ norms for some fixed $q\in [2,+\infty)$, we replace \eqref{e:dtvconv} by the quantitative assumption that there exist $L>0$ and $h=h(n)\in\N$ such that for any $n\in\N$, 
\begin{equation}\label{e:peudecycles2}
d_{\TV}(\overline{\mathbb{P}}_{G_n}^{(h)},\overline{\mathbb{P}}^{(h)})\leq Lh^{-\frac{q}{2}}.
\end{equation}
The larger we can choose $h$, the stronger our conclusion will be. We explain in Section \ref{s:Nlifts} that \eqref{e:peudecycles2} is satisfied with high probability for random lifts of a fixed base graph, with $h=c\log n$ for some $c>0$ (and $L$ depending on $q$).

\subsubsection{Assumption on the Green functions}
Our second condition concerns Green functions of the limiting rooted trees, in the spirit of the assumption also called \Green\ in \cite{anansabri} (the assumption in \cite{anansabri} is stronger, in the sense that if it holds then our assumption holds, see Remark \ref{r:comparwithAS}). Given a rooted graph $(T;o)$, we denote by $R^T_{oo}$ its  Green function evaluated at the root:
$$
R^T_{oo}(z)=\langle \delta_o, (A(T)-z\Id )^{-1}\delta_o\rangle
$$
where $A(T)$ is the adjacency matrix of $T$.
Let $I_1\subset \R$ and $q\in [2,+\infty)$. We assume:

\Green\ There holds 
\begin{equation}\label{e:stronggreen}
\sup_{\lambda\in I_1, \eta\in(0,1)} \mathbb{E}_{(T;o)\sim \overline{\mathbb{P}}}\left(\left(\Im R^T_{oo}(\lambda+i\eta)\right)^{\frac{q}{2}}+|R^T_{oo}(\lambda+i\eta)|^2\right)<+\infty.
\end{equation}

As explained in Section \ref{s:delocunderGreen} and Appendix \ref{a:localweak}, the assumption \BSCT\ implies that the spectral measures $\mu^{G_n}$ converge as $n\rightarrow+\infty$ toward a measure $\overline{\mu}$, while \Green\ implies that $\overline{\mu}$ is absolutely continuous in $I_1$ (but we actually need the full strength of \eqref{e:stronggreen}, and not only this consequence).

We show in Section \ref{s:Nlifts} that \Green\ is satisfied when $\overline{\mathbb{P}}$ is obtained as the local weak limit of random lifts of a base graph. In this case, $\overline{\mathbb{P}}$ is supported on trees of finite cone type, see Section \ref{s:Nlifts} for a definition.

\begin{remark}\label{r:comparwithAS}
Our condition \textbf{(Green)} is weaker than the condition \textbf{(Green)} considered in \cite{anansabri}. Indeed, it is mentioned in \cite[Remark A.4]{anansabri} that the condition \textbf{(Green)} in \cite{anansabri} implies that for any $s>0$, 
$$
\sup_{\lambda\in I_1,\eta\in (0,1)} \mathbb{E}_{(T,o)\sim \overline{\mathbb{P}}}(|R_{oo}^T(\lambda+i\eta)|^s)<+\infty.
$$
In particular, for $s=2$ and $s=q/2$, this implies our condition \textbf{(Green)}.
\end{remark}

\subsubsection{Statement of the result}
In the setting introduced in Section \ref{s:feshortloops}, our main result reads as follows.
\begin{theorem}\label{t:Linftygreen} Let $q\in[2,+\infty)$ and $(G_n)_{n\in \N}$ be a sequence of graphs with local weak limit $\overline{\mathbb{P}}$ supported on the set of rooted trees. Assume that there exist $L,h_0>0$ such that for any $n\in\N$, \eqref{e:peudecycles2} holds for some $h=h(n)\geq h_0$.  Let $I_1$ be a bounded open set where \Green\ is satisfied for this $q$, and let $c_0>0$ such that $\overline{\mu}$ has density $\geq c_0>0$ in $I_1$. 
Then there exist $C,C'>0$ (depending on $L,h_0,c_0$) such that for any $\Lambda>0$,  any $n\in\N$ and any interval $I\subset I_1$  of length at least $C(\log h)/h$ there holds 
$$
\mathbb{P}_I\left(\|u\|_{L^q}\geq \Lambda C'n^{\frac1q-\frac12}\right) \leq \Lambda^{-q}
$$ 
where $u\sim \mathbb{P}_I$, i.e., $u$ is a random approximate eigenvector of the adjacency matrix of $G_n$.
\end{theorem}
In other words, most approximate eigenvectors spectrally localized in $I$ are optimally delocalized in $L^q$ norm. Note that compared to \cite{anansabri}, we do not need the condition that the graph is an expander. In some applications, for instance for random lifts of a fixed base graph, $h(n)$ may be taken as large as $c\log n$ for some fixed $c>0$, and $N(I)$ in this case is of order $n\frac{\log\log n}{\log n}=o(n)$. Therefore Theorem \ref{t:Linftygreen} improves over Theorem \ref{t:delocgeneralsituation} when \Green\ and \eqref{e:peudecycles2} hold.

In Section \ref{s:Nlifts} we show that Theorem \ref{t:Linftygreen} applies to random lifts of a fixed base graph. In this setting we even obtain optimal $L^\infty$ bounds on approximate eigenvectors (see Theorem \ref{t:randomliftinfty}).

Matrices of size $n\times n$ with $d\gg \log n$  standard Gaussian entries per row and column (other entries being set to $0$) are another example where optimal $L^q$ delocalization can be proved for most approximate eigenvectors, obtained as linear combinations of eigenvectors corresponding to the top $N(I)=o(n)$ eigenvalues, see \cite[Section 4]{vanhandel}.


\subsection{Proof techniques and a measure of delocalization}\label{s:technique}
We use the same methodology for proving Theorems \ref{t:trans}, \ref{t:delocgeneralsituation} and \ref{t:Linftygreen}. Consider $H$ a symmetric $n\times n$ matrix, which is the adjacency matrix of a graph for Theorems \ref{t:trans} and \ref{t:Linftygreen}. We fix a subset $I\subset \R$ and consider $E_I$ given by \eqref{e:EI}. We denote by $\Pi_I$ the orthogonal projector onto $E_I$, which has a kernel $\Pi_I(\cdot,\cdot)$ given by
\begin{equation}\label{e:specproj}
\Pi_I(i,j)=\sum_{\lambda_k\in I} \psi_{\lambda_k}(i)\psi_{\lambda_k}(j)
\end{equation}
where $(\psi_{\lambda_k})_{1\leq k\leq n}$ is any orthonormal basis of $\R^n$ composed of eigenvectors $\psi_{\lambda_k}$ of $H$ with associated eigenvalues $\lambda_k$, and $i,j$ denote the coordinates in the canonical basis of $\R^n$. It is important to notice that $\Pi_I(i,j)$ does not depend on the choice of the orthonormal basis.
Our proofs are based on a detailed study of the quantity
\begin{equation}\label{e:PiIondiag}
\sum_{i\in[n]}\Pi_I(i,i)^{q/2}=\Bigl\|\sum_{\lambda_k\in I}\psi_{\lambda_k}^2\Bigr\|_{L^{q/2}}^{q/2}
\end{equation}
for $q\in (2,+\infty]$. We use multiple times the fact, already used for instance in \cite[Proposition 3.1]{vanhandel}, that if good upper bounds on \eqref{e:PiIondiag} are known, then most linear combinations of the $\psi_{\lambda_k}$, $\lambda_k\in I$, are delocalized in the $L^q$ sense. Several versions of this fact are proved in Sections \ref{s:vertex} and \ref{s:results}. 

If $H$ is the adjacency matrix of a vertex-transitive graph, then \eqref{e:PiIondiag} is explicit and small, even when $I$ is reduced to a singleton. In this case, linear combinations of the $\psi_{\lambda_k}$, $\lambda_k\in I$, are also true eigenvectors and we are able to prove that most eigenvectors are delocalized, see Theorem \ref{t:trans}.  If we are dealing with a general symmetric matrix $H$, then \eqref{e:PiIondiag} is small as soon as $I$ contains sufficiently many (not necessarily distinct) eigenvalues, which implies Theorem \ref{t:delocgeneralsituation}. Finally, when $H$ is the adjacency matrix of a graph $G$ with few short loops, we estimate \eqref{e:PiIondiag} by comparing the resolvent of $H$ with that of trees arising in the universal cover of $G$.

The quantity \eqref{e:PiIondiag} is an interesting measure of delocalization. Compared to the averaged participation ratio (considered for instance in Section 5 of \cite{bordenave})
$$
APR_q(I)=\frac{1}{N(I)}\sum_{\lambda_k\in I} \sum_{i\in[n]} |\psi_{\lambda_k}(i)|^q
$$
the quantity \eqref{e:PiIondiag} does not depend on the choice of the eigenbasis $(\psi_{\lambda_k})_{k\in[n]}$. Notice that \eqref{e:PiIondiag} controls the averaged participation ratio:
$$
APR_q(I)= \frac{1}{N(I)}\sum_{i\in [n]}\sum_{\lambda_k\in I} |\psi_{\lambda_k}(i)|^q\leq  \frac{1}{N(I)}\sum_{i\in [n]}\Bigl(\sum_{\lambda_k\in I} |\psi_{\lambda_k}(i)|^2\Bigr)^{q/2}=\frac{1}{N(I)}\Bigl\|\sum_{\lambda_k\in I}\psi_{\lambda_k}^2\Bigr\|_{L^{q/2}}^{q/2}.
$$


\subsection{Related results}\label{s:biblio}
In this section we provide a very brief (and thus necessarily very incomplete) overview of the literature on delocalization of eigenvectors of graphs and matrices, mostly focusing on papers related to ours.

\subsubsection{Erdös-Rényi and regular graphs.} Very strong delocalization results in terms of $L^\infty$ norms have been proved for the eigenvectors of the adjacency matrix of Erdös–Rényi graphs and for random regular graphs, see for instance \cite{erdosrenyi}, \cite{bauer}, \cite{bauer2}. In a different direction, \cite{brookslindenstrauss} proves that in a regular graph with few short cycles, any subset of vertices supporting $\varepsilon$ of the $L^2$ mass of an eigenvector must be large, and \cite{ananlemasson} proves a quantum ergodicity result for expander regular graphs of fixed degree with few short cycles.

\subsubsection{Wigner and Lévy matrices.} Eigenvectors of Wigner matrices have been studied extensively: we only mention \cite{erdosschleinyau}  for sharp bounds on the $L^\infty$ norms of eigenvectors, and \cite{bourgadeyau}, which proves asymptotic normality of eigenvectors for generalized Wigner matrices and a probabilistic version of quantum unique ergodicity.

Very interestingly for us, the papers \cite{bordenaveguionnet}, \cite{bordenaveguionnet2} and \cite{aggarwal} (see Appendix 17), devoted to the study of eigenvectors of Lévy matrices, provide many insights about the role of the spectral projector in the study of localization/delocalization. 

\subsubsection{Cayley graphs.} The papers \cite{zhao}, \cite{mageethomas}, \cite{naor} are concerned with delocalization properties of eigenvectors of Cayley graphs, which are particular vertex-transitive (hence, highly symmetric) graphs constructed via generators of groups. We provide a detailed account on these three papers at the beginning of Section \ref{s:symmetricgraphs}, and revisit some of their results with a different proof technique well-suited for generalizations. Let us also mention that \textit{eigenvalues} of Cayley graphs are also an active research subject, see \cite{eigenvalues} for a recent survey.

\subsubsection{``Non-homogeneous" graphs and matrices.} Many fewer papers are concerned with properties of eigenvectors on ``non-homogeneous" graphs and matrices. In the paper \cite{anansabri} (see also \cite{anansabrianderson} and \cite{anansabrisurvey}), it is proved that for a sequence of finite graphs endowed with discrete Schrödinger operators, assumed to have few short loops and to be an expander, absolutely continuous spectrum for the weak limit of the sequence (under the form of a control of the Green function) implies quantum ergodicity: spectral delocalization implies spatial delocalization. Under the same assumptions, we prove in Section \ref{s:delocunderGreen} a strong estimate on the $L^q$-norms (which \cite{anansabri} does not give), but only for \textit{most approximate eigenvectors}, and not for \textit{exact eigenvectors} as in \cite{anansabri}. The largest part of the literature in the field is devoted to the study of \textit{exact eigenvectors}, for which the tools of the present paper are too rough (except for the vertex-transitive case and its generalizations).

The paper \cite{lemassonsabri} provides $L^q$-bounds for eigenvectors of Schrödinger operators on large, possibly irregular, finite graphs. These bounds are somehow orthogonal to ours: they are far from being sharp but again, they work for \textit{exact eigenvectors}, whereas our bounds are much sharper but work only for \textit{most approximate eigenvectors}.

Let us also mention the recent paper \cite{vanhandel}, where the authors consider $n\times n$ self-adjoint Gaussian random matrices with $d$ nonzero entries per row. When $d\gg \log n$, they construct a delocalized \textit{approximate top eigenvector} by taking a random superposition of many exact eigenvectors near the edge of the spectrum, and they highlight the fact that delocalization properties of approximate eigenvectors are more universal than those of exact eigenvectors. Section \ref{s:delocunderGreen} in the present paper develops this idea in another direction, for graphs with few short loops with a control on the resolvent.

We finally mention \cite{bordenave2} concerning delocalization of eigenvectors in percolation graphs and the survey \cite{bordenave} from which we took inspiration for our Section \ref{s:delocunderGreen}.

\subsubsection{Compact Riemannian manifolds.} Our paper borrows several techniques from papers concerned with delocalization of eigenfunctions of the Laplacian on compact Riemannian manifolds. Zelditch was the first to notice in \cite{zelditch1} that although there exist localized bases of eigenfunctions of the Laplacian on the sphere, almost any\footnote{for natural probability measures on the space of orthonormal eigenbases} eigenbasis on the sphere is quantum ergodic, i.e., high frequency eigenfunctions are totally delocalized. In the same spirit, the first author and Lebeau proved in \cite{burqlebeau} that almost every Hilbert base of $L^2(\mathbb{S}^d)$ made of $L^2(\mathbb{S}^d)$-normalized spherical harmonics has all its elements uniformly bounded in any $L^q(\mathbb{S}^d)$ space ($q<+\infty$). In \cite{burqlebeau}, this result has been extended to arbitrary manifolds, to the price of considering only approximate eigenfunctions (see also \cite{zelditch2}, \cite{burqlebeauproc}, and \cite{han}, \cite{hantacy}, \cite{ireland}, \cite{ireland2} for more recent developments).

\subsection{Organization of the paper} In Section \ref{s:preliminary} we prove basic results regarding random functions on $\mathbb{S}_I$ picked according to the probability $\mathbb{P}_I$. Building upon this, we prove in Section \ref{s:symmetricgraphs} Theorems \ref{t:trans} and \ref{t:smallscale} about vertex-transitive graphs, and extend Theorem \ref{t:trans}   in Section \ref{s:product} to the case of products of graphs, one of which is vertex-transitive. The proofs are not based on representation theory and are all particularly elementary. In Section \ref{s:results} we use the same arguments, sketched in Section \ref{s:technique}, to prove Theorem \ref{t:delocgeneralsituation}. Finally, Section \ref{s:delocunderGreen} is devoted to the proof of Theorem \ref{t:Linftygreen}. In the appendix, we have gathered several useful results concerning concentration of measure, bounds on spectral measures, and Benjamini-Schramm convergence.

\subsection{Acknowledgments} We thank Charles Bordenave, Mostafa Sabri, Laura Shou, Ramon van Handel and Yufei Zhao for answering our questions related to this project. We are also grateful to two anonymous referees for their careful reading and their questions. This research was supported by the European research Council (ERC) under the European Union’s Horizon 2020 research and innovation programme (Grant agreement 101097172 - GEOEDP).


\section{Preliminary computations} \label{s:preliminary}

 This section is devoted to proving basic results regarding random functions on $\mathbb{S}_I$ picked according to the probability $\mathbb{P}_I$ introduced in Section \ref{s:delocmatrix}. We keep the framework of the introduction: we fix a symmetric $n\times n$ real-valued matrix $H$ and a subset $I\subset \R$, and we denote by $N(I)$ the number of eigenvalues of $H$ in $I$. We recall that the spectral projector $\Pi_I(\cdot,\cdot)$ has been introduced in \eqref{e:specproj}. It does not depend on the choice of an eigenbasis, but in the sequel it will nevertheless be convenient to fix an orthonormal basis of eigenvectors $(\psi_{\lambda_k})_{k\in[n]}$ of $H$.

Finally, we set $\dproj(x)=\Pi_I(x,x)$.
\begin{lemma}\label{l:u(x)grand}
Assume $N(I)\geq 2$. Let $u$ be the random function given by \eqref{e:randomu}. Then for any $t\geq 0$ and $x\in[n]$,
\begin{equation}\label{e:probaplusgrand}
\mathbb{P}_I(|u(x)|>t) =\mathbf{1}_{0\leq t <\dproj(x)^{1/2}}
 	2\frac{\Gamma\left(\frac{N(I)}{2}\right)}{\Gamma\left(\frac{N(I)-1}{2}\right)\Gamma\left(\frac12\right)}\int_0^{\theta_t} (\sin \varphi)^{N(I)-2} d\varphi 
\end{equation}
where $\theta_t\in[0,\pi/2]$ is the unique solution to $\cos \theta_t=t\dproj(x)^{-1/2}$, and $\Gamma$ denotes the Euler Gamma function.
\end{lemma}
\begin{proof}
We set
$$
v_I(x)=\frac{1}{\dproj(x)^{1/2}}(\psi_{\lambda_k}(x))_{\lambda_k\in I}
$$
which is an element of $\mathbb{S}_I$. There holds
$$
u(x)=\sum_{\lambda_k\in I}z_{\lambda_k} \psi_{\lambda_k}(x)=\dproj(x)^{1/2} z\cdot v_I(x)
$$
where $z$ is a random vector whose law is uniform over $\mathbb{S}_I$. In particular $|u(x)|\leq \dproj(x)^{1/2}$, which establishes \eqref{e:probaplusgrand} for $t\geq \dproj(x)^{1/2}$.
If  $0\leq t< \dproj(x)^{1/2}$, using Proposition \ref{e:concmeasphi} we get
$$
\mathbb{P}_I(|u(x)|>t)=\mathbb{P}_I(|z\cdot v_I(x)|>t\dproj(x)^{-1/2})=2\frac{\Gamma\left(\frac{N(I)}{2}\right)}{\Gamma\left(\frac{N(I)-1}{2}\right)\Gamma\left(\frac12\right)}\int_0^{\theta_t} (\sin\varphi)^{N(I)-2} d\varphi
$$
where $\theta_t\in[0,\pi/2]$ is the unique solution to $\cos \theta_t=t\dproj(x)^{-1/2}$. 
\end{proof}

\begin{proposition}
There exists $C>0$ such that for any $2\leq p\leq q<+\infty$ and any $I\subset \R$,
\begin{equation}\label{e:exptoq}
\mathbb{E}\left(\|u\|_{L^p}^q\right)^{1/q}\leq C\sqrt{q} \left(\frac{\|\dproj\|_{L^{p/2}}}{N(I)}\right)^{1/2}
\end{equation}
Moreover, for any $K>0$ there exists $C_K>0$ such that for any $2\leq p\leq q<+\infty$ and any $I$ with $N(I)\leq K$, there holds
\begin{equation}\label{e:bddmult}
\mathbb{E}\left(\|u\|_{L^p}^q\right)^{1/q}\leq C_K \|\dproj\|_{L^{p/2}}^{1/2}.
\end{equation}
\end{proposition}
\begin{proof}
For $N(I)=1$, the result is straighforward. In the sequel we assume $N(I)\geq 2$. Fix $x\in V$. We have
\begin{align}
\mathbb{E}(|u(x)|^q)&=q\int_0^{\infty}t^{q-1}\mathbb{P}_I(|u(x)|>t)dt\nonumber\\
&=2\frac{\Gamma\left(\frac{N(I)}{2}\right)}{\Gamma\left(\frac{N(I)-1}{2}\right)\Gamma\left(\frac12\right)}q\dproj(x)^{q/2}\int_0^{\pi/2}(\cos\theta)^{q-1} \sin \theta \int_0^{\theta} (\sin\varphi)^{N(I)-2} d\varphi d\theta\nonumber\\
&= \frac{\Gamma\left(\frac{q+1}{2}\right)\Gamma\left(\frac{N(I)}{2}\right)}{\Gamma\left(\frac{q+N(I)}{2}\right)\Gamma\left(\frac12\right)}\dproj(x)^{q/2} \label{e:Gamma}
\end{align}
where from first to second line we made the change of variables $t=\dproj(x)^{1/2}\cos(\theta)$ and used Lemma \ref{l:u(x)grand}, and from second to third line we used Fubini's theorem and identities involving the beta function to compute the integrals. Then, using Stirling's approximation, we notice that
\begin{equation}\label{e:stirling}
\left(\frac{\Gamma\bigl(\frac{q+1}{2}\bigr)}{\Gamma\bigl(\frac12\bigr)}\right)^{\frac1q}\leq C\sqrt{q}, \qquad \left(\frac{\Gamma\bigl(\frac{N(I)}{2}\bigr)}{\Gamma\bigl(\frac{q+N(I)}{2}\bigr)}\right)^{\frac1q}\leq \frac{C}{N(I)^{\frac12}}
\end{equation}
for some universal constant $C$ (independent of $q\geq 2$ and $N(I)\geq 1$).
We conclude, using Minkowski inequality (recall $p\leq q$), that
\begin{equation}
\left[\mathbb{E}\left(\|u\|_{L^p}^q\right)\right]^{1/q}\leq \|\left(\mathbb{E}|u(\cdot)|^q\right)^{\frac 1 q} \|_{L^p}\\
\leq  C \sqrt{q}\frac{1}{N(I)^{1/2}} \|\dproj(\cdot)^{1/2}\|_{L^p} \end{equation}
which is exactly \eqref{e:exptoq}. And \eqref{e:bddmult} is deduced directly from \eqref{e:Gamma} and Minkowski's inequality without using \eqref{e:stirling}.
\end{proof}



\section{Delocalized eigenbases for transitive graphs and their generalizations}\label{s:symmetricgraphs}
Our goal in this section is to give short proofs of Theorems \ref{t:trans} and \ref{t:smallscale}. Our arguments extend to products of graphs, when one of the two graphs in the product is vertex-transitive.

Given a graph $G$, we first explain how to pick an $L^2$-orthonormal basis of eigenvectors of its adjacency matrix $A_G$ uniformly at random. We denote by $E_1,\ldots,E_K$ the distinct eigenspaces of $A_G$, where $E_k$ has dimension $m_k$. We identify the space of $L^2$-orthonormal bases of $E_k$ with the orthogonal group $O(m_k)$ and endow this space with its Haar measure denoted by $\nu_k$. There holds
\begin{equation}\label{e:eigendecompo}
L^2(V)=\bigoplus_{k\in[K]} E_k
\end{equation}
where $V$ is the set of vertices of $G$. The set of orthonormal eigenbases of $L^2(V)$ compatible with the decomposition \eqref{e:eigendecompo} is
$$
\mathcal{B}= O(m_1)\times \ldots\times O(m_K)
$$
and it is endowed with the product probability measure $\nu=\otimes_k \nu_k$.

\subsection{Proof of Theorem \ref{t:trans}}\label{s:vertex}

Fix an arbitrary set $I\subset \R$. For any $x\in V$,
\begin{equation}\label{e:cstsum}
\frac{\dproj(x)}{N(I)}=\frac{1}{n}
\end{equation}
since $\dproj(x)$ does not depend on $x$ and the sum over $x$ is equal to $1$. We plug \eqref{e:cstsum} into \eqref{e:exptoq} (with $p=q$): for the random function $u$ given by \eqref{e:randomu} we obtain
\begin{equation}\label{e:bornesup}
\mathbb{E}\left(\|u\|_{L^q}^q\right)\leq \left(C \sqrt{q} n^{\frac1q -\frac12}\right)^{q}
\end{equation} 
To prove \eqref{e:Lqbasis}, we observe that 
for any $f:V\rightarrow\mathbb{C}$, $\|f\|_{L^\infty}\leq \|f\|_{L^q}$.
We set $q=\log n$ and we apply the Markov inequality to get for any $\Lambda>0$
\begin{multline}
 \mathbb{P}_I \left( \| u \|_{L^\infty} \geq \Lambda C\log (n)^{\frac12} n^{-\frac 1 2}\right) = \mathbb{P}_I \left( \| u \|^q _{L^\infty} \geq \Lambda^q C^q \log (n)^{\frac{q}{2}} n^{-\frac q 2}\right)\\
\leq \mathbb{P}_I \left( \| u \|^q _{L^q} \geq \Lambda^q C^q \log (n)^{\frac{q}{2}} n^{-\frac q 2}\right)\leq \frac{ \left(C \sqrt{q} n^{\frac1q - \frac 1 2} \right)^q} {\Lambda^q C^q \log (n)^{\frac{q}{2}}n^{-\frac q 2}}=n^{1-\log(\Lambda)} \label{e:Linftytransitive}
\end{multline}

Fix $k\in[K]$ and set $I$ to be the singleton containing the eigenvalue corresponding to the eigenspace $E_k$. In particular $N(I)=m_k$. For any $\ell_0\in [m_k]$, the map 
$$
O(m_k)\ni (b_\ell)_{\ell\in[m_k]} \mapsto b_{\ell_0}\in\mathbb{S}_{\lambda_k}
$$
sends the measure $\nu_k$ to the measure $\mathbb{P}_{\{\lambda_k\}}$ and consequently, according to \eqref{e:Linftytransitive},
$$
\nu_k\left(\left\{(b_\ell)_{\ell\in[m_k]}\in O(m_k); \| b_{\ell_0}\|_{L^\infty} \geq \Lambda C\log (n)^{\frac12} n^{-\frac 1 2}\right\}\right)\leq n^{1-\log(\Lambda)}
$$
We deduce by the union bound
$$
\nu_k\left(\left\{(b_\ell)_{\ell\in[m_k]}\in O(m_k); \exists \ell_0\in[m_k], \, \| b_{\ell_0}\|_{L^\infty} \geq  \Lambda C \log (n)^{\frac12} n^{-\frac 1 2}\right\}\right)\leq  m_kn^{1-\log(\Lambda)}
$$
and finally
$$
\nu\left(\left\{(b_{k,\ell})_{k\in[K], \ell\in[m_k]}\in\mathcal{B};\ \forall k\in[K],\ell\in[m_k], \; \| b_{k,\ell}\|_{L^\infty} \leq  \Lambda C \log (n)^{\frac12} n^{-\frac 1 2}\right\}\right)\geq 1-n^{2-\log(\Lambda)}
$$
follows by union bound over $k\in[K]$, which proves \eqref{e:Lqbasis}. The proof of \eqref{e:Lqbasisq} follows exactly the same lines, except that in \eqref{e:bornesup} and \eqref{e:Linftytransitive} we need to use an arbitrary $q$ independent of $n$, and not $q=\log(n)$.

\begin{remark}\label{r:bdded}
 One can also prove a deterministic bound in terms of the maximal multiplicity $M_n$ of the adjacency matrix $A_G$: any $L^2$-normalized eigenvector $u$ of $A_G$ verifies
\begin{equation}\label{e:optimized}
\forall q\in [2,+\infty], \qquad \|u\|_{L^q}\leq M_n^{\frac12}n^{\frac1q-\frac12}
\end{equation}
(with $q=+\infty$ allowed). Indeed, setting $I=\{\lambda\}$ where $\lambda$ is the eigenvalue associated to $u$, we have for any $x\in V$
$$
|u(x)|=\Bigl|\sum_{\lambda_k=\lambda}a_{\lambda_k}\psi_{\lambda_k}(x)\Bigr|\leq \dproj(x)^{1/2}= \left(\frac{N(I)}{n}\right)^{\frac12}\leq \left(\frac{M_n}{n}\right)^{\frac12}
$$
since $\dproj(x)$ does not depend on $x$, which implies \eqref{e:optimized}. In particular, if $M_n\ll \log(n)$ this improves \eqref{e:Lqbasis}. This is for instance the case for cycle graphs, since $M_n=2$.
\end{remark}

\subsection{Generalization to products of graphs}\label{s:product}
Theorem \ref{t:trans} may be generalized at no cost to products of graphs of all kinds, as soon as one of the two graphs in the product is vertex-transitive; the results below are meaningful when this vertex-transitive graph has a large number of vertices.
\begin{definition} 
A graph product of two graphs $G$ and $H$ with sets of vertices denoted by $V(G)$ and $V(H)$ is a new graph whose vertex set is $V(G)\times V(H)$ and where, for any two vertices $(g,h)$ and $(g',h')$ in the product, the adjacency of those two vertices is determined entirely by the adjacency (or equality, or non-adjacency) of $g$ and $g'$, and that of $h$ and $h'$. 
\end{definition}
\begin{remark}
To define a graph product, there are $3\cdot 3-1=8$ different choices to make\footnote{for instance, one of them is to decide if we put an edge between $(g,h)$ and $(g',h')$ when $g\sim g'$ and $h= h'$. The case where $g=g'$ and $h=h'$ is not considered.} and thus there are $2^8=256$ different types of graph products that can be defined. The most commonly used are the Cartesian product, the lexicographic product, the strong product and the tensor product.
\end{remark}


In this section, the symbol $\times$ denotes one of the 256 possible notions of products of graphs.
\begin{definition} 
We say that a graph is a product of type $(\ell,m)$ if it is equal to $G\times H$ where $G$ is a vertex-transitive graph with $\ell$ vertices, and $H$ is an arbitrary graph with $m$ vertices.
\end{definition}
The following result extends Theorem \ref{t:trans} to products of graphs, with bounds which depend on $\ell$ instead of $n$.
\begin{theorem}\label{t:product}
There exists $C>0$ such that the following holds.  Let $G\times H$ be a product graph of type $(\ell,m)$, with $n=\ell m$ vertices. Then for any $\Lambda>0$, with probability $\geq 1-n\ell^{1-\log(\Lambda)}$ on the choice of an orthonormal eigenbasis $\mathcal{B}$ of $A_G$, any $u\in\mathcal{B}$ verifies 
\begin{equation}\label{e:Linftylogell}
\|u\|_{L^\infty}\leq \Lambda C\sqrt{\frac{\log(\ell)}{\ell}}.
\end{equation}
Also, there exists $C>0$ such that for any $q\in[2,+\infty)$ and any $\Lambda>0$, with probability $\geq 1-n\Lambda^{-q}$ on the choice of an element  $\mathcal{B}\in \mathcal{B}(G)$ picked following $\nu$, any $u\in\mathcal{B}$ verifies
\begin{equation}\label{e:Lqbasisqprod}
\|u\|_{L^q}\leq \Lambda C\sqrt{q} \ell^{\frac1q-\frac12}.
\end{equation}
\end{theorem}
\begin{remark}If instead of choosing an eigenbasis randomly we choose only one eigenvector randomly, we get~\eqref{e:Linftylogell}  with probability $\geq 1-\ell^{1-\log(\Lambda)}$,  and~\eqref{e:Lqbasisqprod} with probability $\geq 1-\Lambda^{-q}$ (i.e., we save a factor $n$  in the proof since we avoid one union bound compared to the case where a whole eigenbasis is picked at random).
\end{remark}
\begin{proof}
We set 
\begin{equation}\label{e:rootineq}
v(x)=\frac{1}{N(I)}\dproj(x).
\end{equation}
We notice that $\sum_{x\in V(G\times H)} v(x)=1$ and that if $x=(g,h)$ and $x'=(g',h)$ with $g,g'\in V(G)$ and $h\in V(H)$, then $v(x)=v(x')$ since $G$ is vertex-transitive. Therefore, $v(x)\leq 1/\ell$ for any $x\in V(G)\times V(H)$.
We denote by $c(h)$ the value of $v(x)$ for $x=(g,h)$ (independent of $g\in G$). We have
\begin{equation*}
\frac{1}{N(I)^{q/2}}\|\dproj\|_{L^{q/2}}^{q/2}=\ell\sum_{h\in H} c(h)^{q/2} \leq \ell\left(\sum_{h\in H} c(h)\right)^{q/2}= \ell^{1-\frac{q}{2}}\left(\sum_{h\in H} \ell c(h)\right)^{q/2}=\ell^{1-\frac{q}{2}}.
\end{equation*}
We deduce from \eqref{e:exptoq} 
$$
\mathbb{E}\left(\|u\|_{L^q}^q\right)\leq \left(C \sqrt{q} \ell^{\frac1q -\frac12}\right)^{q}.
$$
The proof is now exactly the same as in Theorem \ref{t:trans}, with $n$ replaced by $\ell$, and $q=\log \ell$.
\end{proof}


\subsection{Proof of Theorem \ref{t:Gaussianstat}}\label{s:Gauss}
The proof of Theorem \ref{t:Gaussianstat} relies on the general principle that for large collections of high-dimensional data, most one-dimensional projections of the data are approximately Gaussian. This fact has first been proved by Diaconis and Freedman in \cite{diaconis}, and then quantitative bounds have been derived by Meckes \cite{meckes}. Her result reads as follows.
\begin{theorem}\cite[Theorem 3]{meckes} \label{t:meckes}
Let $\{x_i\}_{i=1}^n$ be deterministic vectors in $\R^m$, let $\sigma^2$ be defined by $\frac1n \sum_{i=1}^n |x_i|^2=\sigma^2m$, and assume that
\begin{equation*}
\frac1n \sum_{i=1}^n \left| \sigma^{-2}|x_i|^2-m\right| \leq A \quad \text{and } \quad \forall \theta\in\mathbb{S}^{m-1},\ \frac1n \sum_{i=1}^n \langle \theta,x_i\rangle^2\leq B
\end{equation*}
for some $A,B\geq 0$. Consider the random measure $\mu_n^\theta$ on $\R$ which puts mass $\frac1n$ at each of the points $\langle x_1,\theta\rangle, \ldots,\langle x_n,\theta\rangle$. If $\theta$ is chosen uniformly from $\mathbb{S}^{m-1}$ and 
$$
B\geq \varepsilon \geq \max\left( \left[\frac{3.2^6 \pi B}{\sqrt{m-1}}\right]^{2/5},\frac{2(A+2)}{m-1}\right)
$$
then
$$
\mathbb{P}\left[d_{\rm BL}(\mu_n^\theta,\mathcal{N}(0,\sigma^2))>\varepsilon\right]\leq \frac{c_1\sqrt{B}}{\varepsilon^{3/2}}\exp \left( - \frac{c_2(m-1)\varepsilon^5}{B^2}\right)
$$
with $c_1=48\sqrt{\pi}$, $c_2=3^{-2}2^{-16}$, and $d_{\rm BL}$ denoting the bounded Lipschitz distance.
\end{theorem}

Let us explain how to deduce Theorem \ref{t:Gaussianstat} from Theorem \ref{t:meckes}. We denote by $\psi_1,\ldots,\psi_m$ an orthonormal basis of $E$ and by $\theta$ an element of the unit sphere of $E$ chosen uniformly at random. We also label the vertices as $V=\{1,\ldots,n\}$. For any $i\in V$, we have $x_i=\sqrt{n}(\psi_1(i),\ldots,\psi_m(i))$, and $|x_i|^2=m$ due to \eqref{e:cstsum}. Moreover, $u=\sum_{j=1}^m \theta_j\psi_j\in\mathbb{S}^{n-1}$, therefore $\sqrt{n}u_i=\langle \theta,x_i\rangle$. Then we apply Theorem \ref{t:meckes}. We observe that $\sigma=1$. Also we may take $A=0$ and $B=1$ since
$$
\frac 1 n \sum_{i=1}^n \langle\theta, x_i\rangle^2 = \sum_{i=1}^n \sum_{k,\ell=1}^m \theta_k \theta_\ell \psi_k (i) \psi_\ell(i)
=  \sum_{k,\ell=1}^m \theta_k \theta_\ell  \sum_{i=1}^n \psi_k (i) \psi_\ell(i)= \sum_{k,\ell=1}^m \theta_k \theta_\ell \delta_{k=\ell}  =1.
$$
All in all Theorem \ref{t:meckes} gives exactly Theorem \ref{t:Gaussianstat}.

\subsection{Proof of Theorem \ref{t:smallscale}}
We consider an eigenspace $E_k$ of dimension $m_k$, and we denote by $\psi_1,\ldots,\psi_{m_k}$ an orthonormal basis of $E_k$. As before, $\mathbb{P}_{\{\lambda_k\}}$ is the uniform probability measure on the unit sphere of $E_k$, and $u=\sum_{i=1}^{m_k} z_i\psi_i$ is a random vector following $\mathbb{P}_{\{\lambda_k\}}$ (in particular, $Z=(z_i)_{i\in[m_k]}$ is on the unit sphere of $E_k$). 

Notice that we may assume that each $f_i$ has mean $0$: $\frac1n \sum_{x\in V}f_i(x)=0$. In the sequel we fix $f\in L^2(V)$ with mean $0$. We have
\begin{equation*}
\sum_{x\in V} f(x)u(x)^2=Z^\top BZ
\end{equation*}
where
$$
B=(\beta_{ij})_{1\leq i,j\leq m_k}, \qquad \beta_{ij}=\sum_{x\in V} f(x)\psi_i(x)\psi_j(x).
$$
We notice that 
\begin{equation}\label{e:TrB}
\sum_{i=1}^{m_k}\beta_{ii}=\sum_{x\in V}f(x)\sum_{i=1}^{m_k}\psi_i(x)^2=\frac{m_k}{n}\sum_{x\in V}f(x)=0
\end{equation}
and
\begin{align}
\sum_{i,j=1}^{m_k}\beta_{ij}^2&=\sum_{x,y\in V}f(x)f(y)\left(\sum_{i=1}^{m_k} \psi_i(x)\psi_i(y)\right)^2\leq \sum_{x,y\in V}\frac{f(x)^2+f(y)^2}{2}\left(\sum_{i=1}^{m_k} \psi_i(x)\psi_i(y)\right)^2\nonumber\\
&=\sum_{x,y\in V}f(x)^2\left(\sum_{i=1}^{m_k} \psi_i(x)\psi_i(y)\right)^2=\sum_{x\in V}\sum_{i,j=1}^{m_k}f(x)^2 \psi_i(x)\psi_j(x)\sum_{y\in V}\psi_i(y)\psi_j(y)\nonumber \\
&=\sum_{x\in V}\sum_{i=1}^{m_k}f(x)^2\psi_i(x)^2=\frac{m_k}{n}\|f\|_{L^2}^2\label{e:TrB2}
\end{align}
where we used $\sum_{y\in V}\psi_i(y)\psi_j(y)=\delta_{ij}$.
Since $B$ is symmetric, we write $B=P^TDP$ where $P$ is orthogonal and $D=(d_i)_{1\leq i\leq m_k}$ is diagonal, and we deduce from \eqref{e:TrB}, \eqref{e:TrB2} that
\begin{equation}\label{e:pasbien}
\sum_{i=1}^{m_k} d_i={\rm Tr}(D)={\rm Tr}(B)=0, \qquad \sum_{i=1}^{m_k} d_{i}^2={\rm Tr}(D^2)={\rm Tr}(B^2)\leq \frac{m_k}{n}\|f\|_{L^2}^2.
\end{equation}
Finally, setting $Y=PZ$, which is also uniform on the unit sphere of $E_k$, we have
\begin{equation}\label{e:gaussian}
\sum_{x\in V} f(x)u(x)^2=Z^\top BZ=Y^\top DY.
\end{equation}

We may write $Y_i=\theta_i/\sqrt{\Theta}$ with $\theta_i$ independent standard real normal random variables and 
$$
\Theta=\sum_{i=1}^{m_k} |\theta_i|^2.
$$
We notice that  
\begin{equation}\label{e:Ptheta}
P(\Theta\geq m_k/2)\geq 1-e^{-\frac{m_k}{12}}
\end{equation}
(see \cite[Lemma 5.1]{mageethomas}). Also, applying \cite[Lemma 5.2(i)]{mageethomas} with $C= \frac{m_k}{n}\|f\|_{L^2}^2$ (due to \eqref{e:pasbien}) and $A=\sqrt{C}$, we get that for any $T>0$,
\begin{equation}\label{e:Ptheta2}
\mathbb{P}\left(\theta^\top D\theta\geq Tm_k/2\right)=\mathbb{P}\left(\sum_{i=1}^{m_k} d_i\theta_i^2\geq Tm_k/2\right)\leq 2\left(\frac{Tm_k}{2\sqrt{C}}+1\right)^{\frac12}\exp\left(-\frac{Tm_k}{4\sqrt{C}}\right)
\end{equation}
where $\theta=(\theta_1,\ldots,\theta_{m_k})$. We choose $T=\frac{t\|f\|_{L^2}}{\sqrt{n}}$. Combining \eqref{e:Ptheta} and \eqref{e:Ptheta2} we get 
\begin{equation}\label{e:ameliorer}
\mathbb{P}\left(Y^\top DY\geq \frac{t\|f\|_{L^2}}{\sqrt{n}}\right)\leq  2\left(\frac{t\sqrt{m_k}}{2}+1\right)^{\frac12}\exp\left(-\frac{t\sqrt{m_k}}{4}\right)+\exp\left(-\frac{m_k}{12}\right).
\end{equation}
Recalling \eqref{e:gaussian} and using a union bound over the random choice of the $m_k$ elements forming an orthonormal basis of $E_k$, this concludes the proof.

\section{$L^q$-delocalization of approximate eigenvectors of symmetric matrices} \label{s:results} 
This section is devoted to the proof of Theorem \ref{t:delocgeneralsituation}. We work in the setting of Section \ref{s:delocmatrix}. Our proof is based on the following lemma.
\begin{lemma}\label{l:median}
 Let $u\in\mathbb{S}_I$ be a random vector with law $\mathbb{P}_I$.
\begin{enumerate}
\item There exists $C_1>0$ universal (not depending on $I$) such that for any $q\in(2,+\infty)$,
\begin{equation}
\mathcal{M}_{q,I}\leq C_1\sqrt{q} N(I)^{\frac1q-\frac12}
\end{equation}
where $\mathcal{M}_{q,I}$ denotes the median of the random variable $\|u\|_{L^q}$.
\item Let $q\in (2,+\infty]$. Then, for any $r>0$, 
\begin{equation}\label{e:measconcunbounded}
\mathbb{P}_I(|\|u\|_{L^q}-\mathcal{M}_{q,I}|>r)\leq 4e^{-\frac{N(I)r^2}{2}}.
\end{equation}
\end{enumerate}
\end{lemma}

\begin{proof}
We first prove that for any $q\in[2,+\infty]$
\begin{equation}\label{e:OLq}
\|\dproj\|_{L^{q/2}}\leq N(I)^{2/q}.
\end{equation}
For $j\in [n]$ we denote by $\delta_j$ the vector of $\R^n$ whose only non-zero coordinate is the $j$-th one, whose value is $1$. We fix an orthonormal basis $(\psi_{\lambda_k})_{k\in[n]}$ of eigenvectors of $H_n$. Denoting by $\psi_{\lambda_k}(i)$ the $i$-th coordinate of $\psi_{\lambda_k}$ for $i\in[n]$, we have
\begin{equation}\label{e:super}
\|\dproj\|_{L^\infty}=\max_{j\in[n]}\sum_{\lambda_k\in I}\psi_{\lambda_k}(j)^2\leq \max_{j\in[n]} \sum_{k=1}^n \psi_{\lambda_k}(j)^2=\max_{j\in[n]}\sum_{k=1}^n \langle \delta_j, \psi_{\lambda_k}\rangle^2=\max_{j\in[n]}\|\delta_j\|^2=1.
\end{equation}
We also notice that
\begin{equation}\label{e:L1}
\|\dproj\|_{L^{1}}=\sum_{j\in[n]}\sum_{\lambda_k\in I} \psi_{\lambda_k}(j)^2=N(I).
\end{equation}
Using the interpolation inequality $\|f\|_{L^r}\leq \|f\|_{L^\infty}^{1-\frac1r}\|f\|_{L^1}^{\frac1r}$ with $r=q/2$, we obtain \eqref{e:OLq}.

Point (1) follows from
\begin{align*}
\frac12 \mathcal{M}_{q,I}\leq \mathbb{E}(\|u\|_{L^q})\leq\mathbb{E}(\|u\|_{L^q}^q)^{1/q} \leq C\sqrt{q}N(I)^{\frac{1}{q}-\frac{1}{2}}
\end{align*}
where the first inequality comes from the fact that $\|u\|_{L^q}$ is a non-negative random variable, and the last one from \eqref{e:OLq} plugged into \eqref{e:exptoq}.

We turn to Point (2). Let $F(u)=\|u\|_{L^q}$. We notice that $\|F\|_{\text{lip}}\leq 1$ since
$$
|F(u)-F(v)|\leq \|u-v\|_{L^q}\leq \|u-v\|_{L^2}.
$$
Applying Theorem \ref{t:concentrationmeasure}, we get \eqref{e:measconcunbounded}.
\end{proof}

The proof of Theorem \ref{t:delocgeneralsituation} is now straightforward.
\begin{proof}[Proof of Theorem \ref{t:delocgeneralsituation}]
Point (i) follows by combining Point (1) and Point (2) of Lemma \ref{l:median}, with $r=\Lambda C_1\sqrt{q} N(I)^{\frac1q-\frac12}$ and taking  $C=2C_1$. For Point (ii) we apply Point (i) with $q=\log(N(I))$ and we use the elementary inequality $\| u\|_{L^\infty} \leq \| u\|_{L^q}$.
\end{proof}

\begin{remark}
When $N(I)$ remains bounded as $n\rightarrow +\infty$, Theorem \ref{t:delocgeneralsituation} becomes almost empty. But the method of proof of Theorem \ref{t:delocgeneralsituation} can be straightforwardly adapted to prove smallness of the $L^q$-norm of $u$ with high $\mathbb{P}_I$-probability, under the additional assumption that $\|\Pi_I\|_{L^2\rightarrow L^q}=o(1)$ as $n\rightarrow +\infty$. We do not detail this here. 
\end{remark}




\section{Stronger $L^q$-delocalization for general graphs under Green function bounds}\label{s:delocunderGreen}

The goal of this section is to prove Theorem \ref{t:Linftygreen} and apply it to random lifts of a fixed base graph. The general structure of the proof of Theorem \ref{t:Linftygreen} is borrowed from \cite[Theorem 5.8]{bordenave}, which shows a spectral projector estimate for graphs close to regular graphs. Our proof is an adaptation of this proof to the case of general local weak limits supported on rooted trees (not necessarily regular).

\subsection{Absolute continuity of $\overline{\mu}$}\label{s:abscont}
 Denote by $\mu^{G_n}$ the spectral measure of the adjacency matrix $A(G_n)$, i.e., 
\begin{equation}\label{e:empmeas}
\mu^{G_n}=\frac1n \sum_{k=1}^{|V_n|} \delta_{\lambda_k}
\end{equation} 
where the $\lambda_k$ denote the eigenvalues of $A(G_n)$. Two rooted graphs $(G, o)$ and $(G',o')$ are called equivalent if there is a graph isomorphism $\varphi:G\rightarrow G'$ such that $\varphi(o)=o'$.  If $(G;o)$ is a rooted graph, we denote by $[G;o]$ its equivalence class, and $\mathscr{G}_*$ denotes the set of equivalence classes of connected rooted graphs. Let $\mu^{[G;o]}$ denote the spectral measure of a rooted graph $(G;o)$, defined as the unique probability measure on $\R$ such that
\begin{equation}\label{e:rootedspecmeas}
\forall z, \Im(z)>0, \qquad \langle \delta_o, (A(G)-z\Id)^{-1}\delta_o\rangle=\int_{\R}\frac{1}{\lambda-z} d\mu^{[G;o]}(\lambda).
\end{equation}
According to Proposition \ref{p:limspectralmeasure}, $\mu^{G_n}$ converges to the mean of the empirical measures under $\overline{\mathbb{P}}$, i.e,
\begin{equation}\label{e:overlinemu}
\overline{\mu}=\int_{\mathscr{G}_*} \mu^{[G;o]}d\overline{\mathbb{P}}([G;o]).
\end{equation}
We recall that $\overline{\mathbb{P}}$ is supported on (equivalence classes of) rooted trees.
\begin{proposition}
The measure $\overline{\mu}$ is absolutely continuous in $I_1$, with bounded density.
\end{proposition}
\begin{proof}
Using \Green\, we obtain that for $\overline{\mathbb{P}}$-almost every rooted tree $[T;o]$,  
$$
\liminf_{\eta \downarrow 0} \int_{I_1} |R^T_{oo}(\lambda+i\eta)|^2 d\lambda <+\infty.
$$
In the sequel we work under this $\overline{\mathbb{P}}$-almost sure event. According to \cite[Theorem 4.1]{klein}, the spectral measure $\mu^{[T;o]}$ is absolutely continuous in $I_1$, with density $\rho^{[T;o]}(\lambda)=\frac{1}{\pi}\Im(R^T_{oo}(\lambda+i0))$ with respect to the Lebesgue measure $\ell$ on $\R$.

Denoting by $C>0$ the supremum in the right-hand side of \Green, we get that for any borelian $B\subset I_1$,
$$
\overline{\mu}(B)= \int_{\mathscr{G}_*} \rho^{[T;o]}(B) d\overline{\mathbb{P}}([T;o])\leq \ell(B)\sup_{\lambda\in I_1, \eta\in(0,1)} \mathbb{E}_{(T;o)\sim \overline{\mathbb{P}}}\left(\frac{1}{\pi}\Im R^T_{oo}(\lambda+i\eta)\right)\leq \frac{C^{\frac 2 q}}{\pi}\ell(B)
$$
according to \Green. Hence $\overline{\mu}$ is absolutely continuous, with density bounded by $\frac{C^{\frac2q}}{\pi}$.
\end{proof}

\subsection{Preliminary lemmas} Recall the notation $\overline{\mathbb{P}}_{G_n}^{(h)}$ introduced in \eqref{e:PGnh}. The goal of this subsection is to prove the following proposition.
\begin{proposition}[Spectral projector estimate] \label{p:delocgreen}
Let $(G_n)$ be a family of graphs. We assume that there exist $h_0,L$ such that for some $h=h(n)\geq h_0$, \eqref{e:peudecycles2} holds for any $n\in\N$.
 Let $I_1$ be an open set where \Green\ is satisfied, and let $c_0>0$ such that $\overline{\mu}$ has density $\geq c_0>0$ in $I_1$. 
Then there exist $C,C'>0$ (depending on $c_0, L,h_0$) such that for any interval $I$ of length at least $C(\log h)/h$ such that $I\subset I_1$,
\begin{equation}\label{e:delocexpbstgreen}
\frac{\|\dproj\|_{L^{q/2}}}{N(I)}\leq C'n^{\frac{2}{q} -1}
\end{equation}
\end{proposition}


In the sequel we set $\zeta=e^2\pi$. For any (finite or infinite) graph $H$ whose adjacency operator $A(H)$ is essentially self-adjoint, and for any $z\in\mathbb{C}$, we set $R^H(z)=(A(H)-z\Id)^{-1}$. We start by recalling the following result, which states that if two rooted graphs are isomorphic up to distance $h$ of their root, then their resolvents are $h^{-1}$-close. 
\begin{lemma} \label{l:closeresolvent}
For $i=1,2$, let $(G_i,o)$ be a  rooted graph, and assume that the adjacency operator $A(G_i)$ is essentially self-adjoint. Assume further that $(G_1;o)_h$ and $(G_2;o)_h$ are isomorphic for some $h\in\mathbb{N}$ and that $\|A(G_i)\|\leq b$ for $\|\cdot\|$ the operator norm. Then for any $z\in\mathbb{C}$ such that $\Im(z)\geq \zeta b \lceil \log(2h)\rceil / 2h$,
$$
|R_{oo}^{G_1}(z)-R_{oo}^{G_2}(z)|\leq \frac{1}{\zeta bh}.
$$
\end{lemma}


For a proof, see e.g. \cite[Corollary 5.5]{bordenave}. The next proposition tells us that the spectral measure $\mu_{G_n}$ of $G_n$ is close to $\overline{\mu}$ for large $n$, at least over not too small intervals $I\subset \R$. Recall that $D$ denotes a uniform upper bound on the degree of the graphs $G_n$, see Section \ref{s:feshortloops}.
\begin{lemma}[Local Kesten-McKay law]\label{l:lockestenmckay}
Let $0<\delta<1$ and assume that there exists $h\geq 1$ such that
$$
\delta\geq \max\left(hd_{\TV}(\overline{\mathbb{P}}^{(h)}_{G_n},\overline{\mathbb{P}}^{(h)}), \frac{1}{h}\right).
$$
Then for any interval $I\subset \R$ of length $|I|\geq \frac{20 D\log(2h)}{h}\left(\frac{1}{\delta}\log \frac{1}{\delta}\right)$ we have
$$
\frac{|\mu^{G_n}(I)-\overline{\mu}(I)|}{|I|}\leq C\delta,
$$
where the constant $C$ only depends on $D$.
\end{lemma}
\begin{proof}[Proof of Lemma \ref{l:lockestenmckay}]
Let $t=\zeta D\lceil \log 2h\rceil / (2h)\leq \frac{20D \log(2h)}{h}$. We denote by $\mathcal{R}_h$ the set of rooted graphs $(H;o')$ with depth $\leq h$. Let $(H;o')\in\mathcal{R}_h$. For simplicity of notation we simply write $H$ instead of $(H;o')$ to denote this rooted graph. We introduce
\begin{equation}\label{e:vho}
V_n(H)=\{x\in G_n \mid (G_n;x)_h\simeq H\}
\end{equation}
and
$$
f_n(z;H)=\frac{1}{|V_n(H)|}\sum_{x\in V_n(H)}R_{xx}^{G_n}(z).
$$
We also introduce the conditional expectation
$$
f(z;H)=\mathbb{E}_{(T;o)\sim \overline{\mathbb{P}}}\left(R_{oo}^T(z) \mid (T;o)_h\simeq H\right)
$$
where $(T;o)_h$ is the rooted tree $T$ cut at distance $h$. We have  from Lemma \ref{l:closeresolvent} if $\Im(z)=t$
$$
|f_n(z;H)-f(z;H)|\leq \frac{1}{\zeta Dh}.
$$
We also have if $\Im(z)=t$, since $\overline{\mathbb{P}}_{G_n}^{(h)}(H)=\frac1n |V_n(H)|$ and $|f_n(z;H)|\leq 1/t$,
\begin{align*}
\Bigl|\frac1n \sum_{x\in V_n} R^{G_n}_{xx}(z)-\mathbb{E}_{(T;o)\sim \overline{\mathbb{P}}}(R^T_{oo}(z))\Bigr|&=\Bigl|\sum_{H\in\mathcal{R}_h} \overline{\mathbb{P}}_{G_n}^{(h)}(H)f_n(z;H)- \overline{\mathbb{P}}^{(h)}(H)f(z;H)\Bigr|\nonumber\\
&\leq \frac{2}{t}d_{\TV}(\overline{\mathbb{P}}^{(h)}_{G_n},\overline{\mathbb{P}}^{(h)})+\sum_{H\in\mathcal{R}_h}|f_n(z;H)-f(z;H)|\overline{\mathbb{P}}^{(h)}(H)\nonumber
\end{align*}
For any finite non-negative measure $\mu$ on $\R$, we set $g_\mu(z)=\int_{\R} \frac{1}{\lambda-z}d\mu(\lambda)$ for $z\in\mathbb{C}$ such that $\Im(z)>0$. Therefore for $\Im(z)=t$
\begin{align*}
|g_{\mu^{G_n}}(z)-g_{\overline{\mu}}(z)|&=\Bigl|\frac1n \sum_{x\in V_n} R^{G_n}_{xx}(z)-\mathbb{E}_{(T;o)\sim \overline{\mathbb{P}}}(R^T_{oo}(z))\Bigr|\leq \frac{2}{t}d_{\TV}(\overline{\mathbb{P}}^{(h)}_{G_n},\overline{\mathbb{P}}^{(h)})+\frac{1}{\zeta Dh}\\
&\leq \frac{4h}{\zeta D}d_{\TV}(\overline{\mathbb{P}}^{(h)}_{G_n},\overline{\mathbb{P}}^{(h)})+\frac{1}{\zeta Dh}.
\end{align*}
By assumption this is bounded above by $C\delta$ for some explicit $C>0$ depending only on $D$. We can apply Lemma \ref{l:measfuncofres} with $K=\R$ and $A$ the adjacency matrix of $G_n$, since the density $\mathbb{E}_{(T;o)\sim \overline{\mathbb{P}}}\left(\frac{1}{\pi}\Im(R^T_{oo}(\lambda+i0))\right)$ of $\overline{\mu}$ is bounded above (according to \Green\ together with Jensen's inequality). This concludes the proof.
\end{proof}
\begin{proof}[Proof of Proposition \ref{p:delocgreen}] 
By assumption on $I_1$, the density of $\overline{\mu}$ is bounded below on $I_1$ by some positive constant $c_0$. Let $C$ be the constant from Lemma \ref{l:lockestenmckay}. Set $\delta_0= c_0/2C$. Without loss of generality, we assume $L\geq 1$ (where $L$ is defined in \eqref{e:peudecycles2}). It follows from \eqref{e:peudecycles2} that
$$
hd_{\TV}(\overline{\mathbb{P}}^{(h)}_{G_n},\overline{\mathbb{P}}^{(h)})\leq L h^{1-\frac{q}{2}}\leq \delta_0
$$
for $h\geq h_1$ where $h_1$ is sufficiently large.
Applying Lemma \ref{l:lockestenmckay} we get that
$$
\frac{\mu_{G_n}(I)}{|I|}\geq c_0-C\delta_0,
$$
for all intervals $I$ of length $|I|\geq c_1\log(2h)/h$ such that $I\subset I_1$, where $c_1= \frac{20 D}{\delta_0}\log \frac{1}{\delta_0}$. In particular, $\mu_{G_n}(I)/|I|\geq c_0/2$ for all these intervals $I$, which implies
\begin{equation}\label{e:minnumber}
N(I)\geq \frac12 c_0n|I|.
\end{equation}
Let $h\geq h_1$ and $t\geq 20 D \log(2h)/h$. In analogy with the notation of the proof of Lemma \ref{l:lockestenmckay} we introduce the function $g_n^{(q)}$ defined for any rooted graph $(H;o')\in\mathcal{R}_h$ (simply denoted by $H$ in the sequel) by
$$
g_n^{(q)}(z;H)=\frac{1}{|V_n(H)|}\sum_{x\in V_n(H)}\left(\Im(R^{G_n}_{xx}(z))\right)^{\frac{q}{2}}, 
$$
where $V_n(H)$ has been introduced in \eqref{e:vho}. We also consider the conditional expectation
$$
g^{(q)}(z;H)=\mathbb{E}_{(T;o)\sim \overline{\mathbb{P}}}\left(\Im(R^{T}_{oo}(z))^{\frac{q}{2}}\mid (T;o)_h\simeq H\right).
$$
We have if $\Im(z)=t$
\begin{multline}
\left|\frac1n \sum_{x\in V_n}\Im(R_{xx}^{G_n}(z))^{\frac{q}{2}}-\mathbb{E}_{(T;o)\sim \overline{\mathbb{P}}} \Im(R^T_{oo}(z))^{\frac{q}{2}}\right|=\left|\sum_{H \in\mathcal{R}_h} \overline{\mathbb{P}}_{G_n}^{(h)}(H)g^{(q)}_n(z;H)- \overline{\mathbb{P}}^{(h)}(H)g^{(q)}(z;H)\right|\\
\leq \frac{2}{t^{\frac{q}{2}}}d_{\TV}(\overline{\mathbb{P}}^{(h)}_{G_n},\overline{\mathbb{P}}^{(h)})+\sum_{H\in \mathcal{R}_h}|g_n^{(q)}(z;H)-g^{(q)}(z;H)|\overline{\mathbb{P}}^{(h)}(H).\label{e:diffexp}
\end{multline}
Let $H\in\mathcal{R}_h$.  For any $x\in V_n(H)$ and $(T;o)$ such that $(T;o)_h\simeq H$ we have according to Lemma \ref{l:closeresolvent}
$$
|\Im(R_{xx}^{G_n}(z))-\Im(R^{T}_{oo}(z))|\leq \frac{1}{\zeta Dh}
$$
and in particular $\Im(R_{xx}^{G_n}(z))\leq \Im(R^{T}_{oo}(z))+1$. Therefore
$$
|\Im(R_{xx}^{G_n}(z))^{\frac{q}{2}}-\Im(R^{T}_{oo}(z))^{\frac{q}{2}}|\leq \frac{1}{\zeta Dh}(1+\Im(R^{T}_{oo}(z)))^{\frac{q}{2}}.
$$
We deduce
\begin{equation}\label{e:q/2-1}
|g_n^{(q)}(z;H)-g^{(q)}(z;H)|\leq \frac{1}{\zeta Dh}\mathbb{E}_{(T;o)\sim \overline{\mathbb{P}}}\left((1+\Im(R^{T}_{oo}(z)))^{\frac{q}{2}}\mid (T;o)_h\simeq H\right).
\end{equation}
Combining \eqref{e:diffexp}, \eqref{e:q/2-1} and \Green\ we deduce
\begin{equation}\label{e:diffimaginary}
\left|\frac1n \sum_{x\in V_n}\Im(R_{xx}(z))^{\frac{q}{2}}-\mathbb{E}_{(T;o)\sim \overline{\mathbb{P}}} \Im(R^T_{oo}(z))^{\frac{q}{2}}\right|\leq d_{\TV}(\overline{\mathbb{P}}^{(h)}_{G_n},\overline{\mathbb{P}}^{(h)})h^{\frac{q}{2}} +\frac{C}{h}
\end{equation}
Using \eqref{e:peudecycles2} and again \Green, we deduce that $\frac1n \sum_{x\in V}(\Im(R_{xx}(z)))^{\frac{q}{2}}$ is bounded by $C'$ for some $C'>0$ depending only on $L,D$. If $I=[\lambda-t,\lambda+t]$ and $z=\lambda+it$, then $\Im((\lambda'-z)^{-1})=t/((\lambda'-\lambda)^2+t^2)\geq (1/2t)\mathbf{1}_{\lambda'\in I}$, therefore
\begin{equation}\label{e:majorbyres}
 \sum_{\lambda_k\in I}|\psi_{\lambda_k}(x)|^2\leq 2t\Im(R_{xx}(z))
\end{equation}
hence
\begin{equation}\label{e:projanansabri}
\|\dproj\|_{L^{q/2}}^{q/2}\leq |I|^{\frac{q}{2}}\sum_{x\in V}(\Im(R_{xx}(z)))^{\frac{q}{2}}\leq C'n|I|^{\frac{q}{2}}.
\end{equation}
Putting this together with \eqref{e:minnumber} we get the result. 
\end{proof}

\subsection{Proof of Theorem \ref{t:Linftygreen}}\label{s:endproofthlinftygreen}
Using \ref{e:exptoq} together with Proposition \ref{p:delocgreen}, we obtain 
$$
\mathbb{E}\left(\|u\|_{L^q}^q\right)\leq \left(C' n^{\frac1q -\frac12}\right)^{q}.
$$
Then, we apply the Markov inequality to get for any $\Lambda>0$
\begin{equation}\label{e:bonnesurprise}
 \mathbb{P}_I \left( \| u \|_{L^q} \geq \Lambda C'n^{\frac1q -\frac12}\right) \leq \left(\frac{ C' n^{\frac1q - \frac 1 2}} {\Lambda C' n^{\frac1q-\frac 1 2}} \right)^q\leq \Lambda^{-q}.
\end{equation}

\subsection{An application: approximate eigenvectors of random lifts}\label{s:Nlifts}
In this section we provide examples of families of graphs where Theorem \ref{t:Linftygreen} applies, namely random lifts of a fixed base graph $G$. Our main result of this section is Theorem \ref{t:randomliftinfty}. Recall the following definition \cite{amitlinial}.
\begin{definition}\label{d:randomnlift}
Given a graph $G$ with vertex set $V$, a random labeled $n$-lift of $G$ is obtained by arbitrarily orienting the edges of $G$, choosing a permutation $\sigma_e\in\mathfrak{S}_n$ for each edge $e$ uniformly and independently at random, and constructing the graph $G_n$ with $n$ vertices $(u,1),\ldots,(u,n)$ for each $u\in V$ and edges $((u,i), (v,\sigma_e(i)))$ whenever $e = (u,v)$ is an oriented edge of $G$.
\end{definition}
In the sequel, a finite graph $G=(V,E)$ is fixed. We prove in Lemma \ref{l:nliftconverge} that $n$-lifts of $G$ converge almost surely in the local weak sense as $n\rightarrow +\infty$ toward a probability $\overline{\mathbb{P}}$ which we now describe. Recall that the universal cover $\widetilde{G}$ of $G$ is a tree of finite cone type, meaning that if one denotes by $\mathcal{C}(v)$ the forward subtree emanating from a vertex $v$ of $\widetilde{G}$, the number of non-isomorphic cones $\mathcal{C}(v)$ as $v$ runs over the vertices of $\widetilde{G}$ is finite. Then, 
$$
\overline{\mathbb{P}}=\frac{1}{|V|}\sum_{x\in V}\delta_{[\widetilde{G},x]}
$$
is a well-defined probability measure.

\begin{lemma}\label{l:nliftconverge}
Let $0<c\leq \frac{1}{8\log(d-1)}$. There exists $C>0$ such that for any $n\in\N$ the probability that a random labeled $n$-lift $G_n$ of $G$ satisfies 
\begin{equation}\label{e:distrandolift}
d_{\TV}(\overline{\mathbb{P}}_{G_n}^{(h)},\overline{\mathbb{P}}^{(h)})\leq \frac{2}{\sqrt{n}}
\end{equation}
for $h=c\log n$ is $\geq 1-C|V|n^{-\frac14}$.
\end{lemma}
\begin{proof}
Let $G_n$ be drawn according to the probability measure  $\mathbb{Q}_n$ on random labeled $n$-lifts introduced in Definition \ref{d:randomnlift}. Fix $0<c\leq\frac{1}{8\log(d-1)}$. Then there exists $C>0$ such that for any $x\in V_n=V\times [n]$ (the vertex set on which all $n$-lifts are built),
\begin{equation}\label{e:vdansuncycle}
\mathbb{Q}_n(A^{(G_n)}_x)\leq C\frac{(d-1)^{2c\log n}}{n}
\end{equation}
where $A^{(G_n)}_x$ denotes the event that there exists at least one cycle in $G_n$ of length $\leq c\log n$ containing $x$. The proof of \eqref{e:vdansuncycle} is almost contained in the proof of \cite[Lemma 27]{bordenavefriedman}. For the sake of completeness, it is detailed in Appendix \ref{a:cycles}.

We denote by $X$ the random variable counting the number of points in $G_n$ which belong to a cycle of length $\leq c\log n$. There holds
$$
\mathbb{Q}_n(X> \sqrt{n})\leq \frac{1}{\sqrt{n}}\mathbb{E}_{\mathbb{Q}_n}(X)=\frac{1}{\sqrt{n}}\sum_{x\in V_n} \mathbb{Q}_n(A_x^{(G_n)})\leq C|V|\frac{(d-1)^{2c\log n}}{\sqrt{n}}\leq C|V|n^{-\frac14}.
$$ 
We observe that for any $G_n\in\mathcal{G}_n$ such that such that $X\leq\sqrt{n}$, \eqref{e:distrandolift} is satisfied. Indeed, 
\begin{equation}\label{e:PhGn}
\overline{\mathbb{P}}^{(h)}=\frac{1}{|V_n|}\sum_{x\in V_n}\delta_{[\widetilde{G},x]_h}
\end{equation}
since $\widetilde{G}$ is also the universal cover of $G_n$. Then we see that the two summands in \eqref{e:PhGn} and \eqref{e:PGnh} coincide except when $A_x^{(G_n)}$ holds, which is the case for a proportion at most $\sqrt{n}/|V_n|$ of the vertices $x\in V_n$, from which we deduce that \eqref{e:distrandolift} holds.
\end{proof}
Let us denote by $A(\widetilde{G})$ the adjacency operator of $\widetilde{G}$, and by ${\rm sp}(A(\widetilde{G}))$ its spectrum. If $G$ is a finite graph with minimal degree $\geq 2$ which is not a cycle, it follows from \cite[Theorem 1.5]{meanquantum} that ${\rm sp}(A(\widetilde{G}))$ has a continuous part. The following statement is concerned with approximate eigenvectors spectrally localized in this continuous part:
\begin{theorem}\label{t:randomliftinfty}
Assume that $G$ is a finite graph with minimal degree $\geq 2$ which is not a cycle. Denote the universal cover of $G$ by $\widetilde{G}$. There exists a family of sets $(I_{c_1,c_2})_{0<c_1\leq c_2<+\infty}$ having the following properties:
\begin{itemize}
\item for any $0<c_1<c_2<+\infty$, there exist $C,C'>0$ such that for any $n\in\N$, any $\Lambda>0$ and any interval $I\subset {\rm sp}(A(\widetilde{G}))\setminus I_{c_1,c_2}$  of length at least $C\frac{\log \log n}{\log n}$ there holds 
\begin{equation}\label{e:presqopt}
\mathbb{P}_I\left(\|u\|_{L^\infty}\geq \Lambda C'(\log n)^2n^{-\frac12}\right) \leq \Lambda^{-\frac{\log n}{2\log \log n}}
\end{equation}
for any $n$-lift $G_n$ of $G$ satisfying \eqref{e:distrandolift}, where $u\sim \mathbb{P}_I$, i.e., $u$ is a random approximate eigenvector of the adjacency matrix of $G_n$.
\item for any $0<c_1\leq c_2<+\infty$, $I_{c_1,c_2}$ is a finite union of open intervals, and $I_{c_1,c_2}$ shrinks to a finite set when $c_1\rightarrow 0$ and $c_2\rightarrow+\infty$. In particular, ${\rm sp}(A(\widetilde{G}))\setminus I_{c_1,c_2}$ is non-empty when $c_1$ is sufficiently close to $0$ and $c_2$ is sufficiently large.
\end{itemize}
\end{theorem}

Notice that thanks to Lemma \ref{l:nliftconverge} we know that Theorem \ref{t:randomliftinfty} applies to random $n$-lifts with probability $\geq 1-C|V|n^{-\frac14}$. We also mention that there is an analogous statement to Theorem \ref{t:randomliftinfty} for $L^q$ norms, $q\in[2,+\infty)$. Finally, it is possible by optimizing slightly the proof of Lemma \ref{l:nliftconverge} and the choice of $q$ in \eqref{e:newq} to replace $(\log n)^2$ in \eqref{e:presqopt} by $(\log n)^\alpha$ for any $\alpha>1/2$.
\begin{proof}[Proof of Theorem \ref{t:randomliftinfty}]
The main idea is that thanks to the strong bound \eqref{e:distrandolift}, our proof of Theorem \ref{t:Linftygreen} still works when $q$ depends on $n$, as long as $q\leq  C\frac{\log n}{\log \log n}$ with $C<1$ (uniform in $n$).
We take $h=c\log n$ where $c=\frac{1}{8\log(d-1)}$ is the constant appearing in Lemma \ref{l:nliftconverge}, and 
\begin{equation}\label{e:newq}
q= \frac{\log n}{2\log\log n}.
\end{equation}
Then according to Lemma \ref{l:nliftconverge}, the inequality \eqref{e:peudecycles2} is satisfied for some $L>0$ (uniformly in $n$). 

We now define  $I_{c_1,c_2}$ for any $0<c_1\leq c_2<+\infty$. For this, we need to rely on the results of \cite{anansabrisurvey} as a black-box: a family of ``resolvent-type" functions $\zeta_j:\mathbb{C}\rightarrow \mathbb{C}$ is introduced (whose dependence in $z\in\mathbb{C}$ is denoted by $\zeta_j^z$), for which the set of $\lambda\in {\rm sp}(A(\widetilde{G}))$ such that $c_1\leq|\Im \zeta_j^{\lambda+i0}|\leq c_2$  for any $j$ is of the form ${\rm sp}(A(\widetilde{G}))\setminus I_{c_1,c_2}$. Here $I_{c_1,c_2}$ is a finite union of open intervals, that shrinks to a finite set when $c_1\rightarrow 0$ and $c_2\rightarrow+\infty$. We refer the reader to part (2) of \cite[Proposition 4.2]{anansabrisurvey}, and to the comments below this proposition. From this, it follows that
\begin{equation}\label{e:stronggreen2}
\sup_{\lambda\in I_1, \eta\in(0,1)} \mathbb{E}_{(T;o)\sim \overline{\mathbb{P}}}\left(\left(\Im R^T_{oo}(\lambda+i\eta)\right)^{\frac{q}{2}}+|R^T_{oo}(\lambda+i\eta)|^2\right)\leq C'^{q/2}
\end{equation}
where $C'$ does not depend on $n$ (see also Remark A.4 in \cite{anansabri}). This replaces \Green.


We claim that Proposition \ref{p:delocgreen} holds (without any change in the statement) if $q$ is taken as \eqref{e:newq}. Indeed, with this $n$-dependent $q$, the proof carries over without any modification until \eqref{e:q/2-1} (included). Instead of \eqref{e:diffimaginary} we get thanks to \eqref{e:stronggreen2}
$$
\left|\frac1n \sum_{x\in V_n}\Im(R_{xx}(z))^{\frac{q}{2}}-\mathbb{E}_{(T;o)\sim \overline{\mathbb{P}}} \Im(R^T_{oo}(z))^{\frac{q}{2}}\right|\leq d_{\TV}(\overline{\mathbb{P}}^{(h)}_{G_n},\overline{\mathbb{P}}^{(h)})h^{\frac{q}{2}} +\frac{C'^{q/2}}{h}
$$
Using \eqref{e:peudecycles2} and the fact that $h=c\log n$, we conclude that  $\frac1n \sum_{x\in V}(\Im(R_{xx}(z)))^{\frac{q}{2}}$ is bounded above by $C''^{\frac{q}{2}}$ for some $C''>0$ depending only on $L,D$. Therefore, Proposition \ref{p:delocgreen} holds with $q$ given by \eqref{e:newq}. Then, as in Section \ref{s:endproofthlinftygreen}, we deduce that \eqref{e:bonnesurprise} holds for $q$ given by \eqref{e:newq}. Finally, using that $\|u\|_{L^\infty}\leq \|u\|_{L^q}$ we get Theorem \ref{t:randomliftinfty}.
\end{proof}

\appendix

\section{Probability calculus on spheres}
Let us denote by $\mu_{d-1}$ the uniform  probability measure on the unit sphere $\mathbb{S}^{d-1}$ of $\R^d$.  Recall the following formula (see \cite[Proposition 5.1]{burqlebeauproc}, which corrects \cite[Appendix A]{burqlebeau}):
\begin{proposition}
Let $d\geq 2$. For any $\theta\in [0,\pi/2]$, 
\begin{equation}\label{e:concmeasphi}
\mu_{d-1}(|x_1|>\cos(\theta))=C_d \int_0^{\theta} \sin^{d-2}(\varphi) d\varphi ,\qquad C_d=2\frac{\Gamma\left(\frac{d}{2}\right)}{\Gamma\left(\frac{d-1}{2}\right)\Gamma\left(\frac12\right)}.
\end{equation}
\end{proposition}
\begin{proof} 
On $\mathbb{S}^{d-1}$ we use the coordinates $(\cos(\varphi),\sin(\varphi)u)$ where $u\in\mathbb{S}^{d-2}$. Then $d\mu_{d-1}(\varphi,u)=\sin(\varphi)^{d-2}d\mu_{d-2}(u)d\varphi$. This yields the integral formula in \eqref{e:concmeasphi}, and there remains to determine $C_d$.
We have
$$
1=C_d\int_0^{\pi/2}\sin^{d-2}(\varphi)d\varphi=\frac{C_d}{2} B((d-1)/2,1/2).
$$
where $B(\cdot,\cdot)$ is the beta function. Using $B(x,y)=\Gamma(x)\Gamma(y)/\Gamma(x+y)$, this gives the value of $C_d$.
\end{proof}

There exist different statements of measure concentration of Lipschitz functions on the sphere in the literature. The one we use in the present paper is the following.
\begin{theorem} \label{t:concentrationmeasure}
Let $f:\mathbb{S}^{d-1}\rightarrow \R$ be a Lipschitz function and define its median value $\mathcal{M}(f)$ by
$$
\mu_{d-1}(f\geq \mathcal{M}(f))\geq \frac12, \qquad \mu_{d-1}(f\leq \mathcal{M}(f))\geq \frac12.
$$
Then for any $r>0$ 
$$
\mu_{d-1}(|f-\mathcal{M}(f)|>r)\leq 4e^{-\frac{dr^2}{2\|f\|_{\rm lip}^2}}.
$$
\end{theorem}
For a proof, see \cite[Theorem 14.3.2]{matousek}, written for $1$-Lipschitz functions. The case of general Lipschitz functions is obtained by considering $f/\|f\|_{\rm lip}$.

\section{Bound on spectral measures using the resolvent}
The following result is proved in \cite[Lemma 5.3]{bordenave}, using \cite[Lemma 3.7]{bordenaveguionnet}.
\begin{lemma} \label{l:measfuncofres}
Let $A\in\mathcal{H}_n(\mathbb{C})$ be a Hermitian matrix with resolvent $R(z)=(A-zI_n)^{-1}$. Let $L\geq 1$, $0<\delta<1/2$, $K$ be an interval of $\R$ and $\mu$ a probability measure on $\R$. Recall that $g_\mu(z)=\int_{\R} \frac{1}{\lambda-z}d\mu(\lambda)$ for $z\in\mathbb{C}$ such that $\Im(z)>0$. We assume that for some $t>0$ and all $\lambda\in K$, either 
$$
\Im g_\mu(\lambda+it)\leq L \qquad \text{or} \qquad \mu\left(\left[\lambda-\frac{t}{2},\lambda+\frac{t}{2}\right]\right)\leq Lt.
$$
We also assume that for all $\lambda\in K$, 
$$
\left|\frac1n {\rm tr}R(\lambda+it)-g_\mu(\lambda+it)\right|\leq \delta.
$$
Then for any interval $I\subset K$ of length $|I|\geq t(\frac{1}{\delta}\log \frac{1}{\delta})$ such that ${\rm dist}(I,K^c)>1/L$ we have
$$
\frac{|\mu_A(I)-\mu(I)|}{|I|}\leq CL\delta
$$
where $C$ is a universal constant and $\mu_A(I)$ is the number of eigenvalues of $A$ belonging to $I$.
\end{lemma}

\section{Local weak convergence and spectral measure} \label{a:localweak} In this appendix we collect known facts on the local weak (also called Benjamini-Schramm) convergence, and we refer the reader to \cite[Appendix A]{anansabri} for details. We define a distance between rooted graphs by
$$
d_{\rm loc}((G,o),(G',o'))=\frac{1}{1+\alpha}, 
$$
$$
\alpha =\sup \{r>0: \exists\, {\rm graph\ isomorphism\ } \varphi: B_G(o, \lfloor r\rfloor)\rightarrow B_{G'}(o',\lfloor r\rfloor)\ {\rm and }\ \varphi(o)=o' \}.
$$

Recall that $\mathscr{G}_*$ denotes the set of equivalence classes of connected rooted graphs under the isomorphism relation. Then $d_{\rm loc}$ turns $\mathscr{G}_*$ into a separable complete metric space. We may thus consider the set of probability measures on $\mathscr{G}_*$, denoted by $\mathcal{P}(\mathscr{G}_*)$. If $(G_n)$ is a sequence of finite graphs, we say that $\overline{\mathbb{P}}\in \mathcal{P}(\mathscr{G}_*)$ is the local weak limit of $(G_n)$ if 
$$
\frac{1}{|V_n|}\sum_{x\in V_n} \delta_{[G_n;x]}
$$
converges weakly-* to $\overline{\mathbb{P}}$ in $\mathcal{P}(\mathscr{G}_*)$ as $n\rightarrow +\infty$, where $V_n$ is the set of vertices of $G_n$, and we recall that $[G_n;x]$ denotes the equivalence class of the rooted graph $(G_n;x)$.

The subset $\mathscr{G}^D_*\subset \mathscr{G}_*$ of equivalence classes $[G;o]$ such that $G$ is of degree bounded by $D$ is compact. It follows that $\mathcal{P}(\mathscr{G}_*^D)$ is compact in the weak-* topology.  Hence, if $\mathcal{C}^D_{\rm fin}$ denotes the set of finite graphs $G$ of degree bounded by $D$, then any sequence $(G_n)\subset \mathcal{C}^D_{\rm fin}$  has a subsequence which converges in the local weak sense to some $\overline{\mathbb{P}}\in\mathcal{P}(\mathscr{G}_*^D)$. If $(G_n)$ satisfies \BST and if a subsequence of $(G_n)$ has a local weak limit $\overline{\mathbb{P}}$, then $\overline{\mathbb{P}}$ must be concentrated on the set of rooted trees with degree bounded by $D$.

Recall the notation $\mu^{G_n}$ and $\mu^{[G;o]}$ introduced respectively in \eqref{e:empmeas} and \eqref{e:rootedspecmeas}.
We recall \cite[Theorem 2.1]{salez}:
\begin{proposition}\label{p:limspectralmeasure}
Suppose a sequence $(G_n)\in \mathcal{C}^D_{\rm fin}$ has a local weak limit $\overline{\mathbb{P}}$. Then $\mu^{G_n}$ converges weakly to $\int_{\mathscr{G}_*^{D}} \mu^{[G,o]} d\overline{\mathbb{P}}([G;o])$.
\end{proposition}

\section{Proof of \eqref{e:vdansuncycle}} \label{a:cycles}
The proof of \eqref{e:vdansuncycle} is contained in the proof of \cite[Lemma 27]{bordenavefriedman} except for the last line. But reading the proof of \cite[Lemma 27]{bordenavefriedman} requires one to be familiar with the notation of \cite{bordenavefriedman}. To keep the present paper self-contained, we repeat here this proof, very mildly modified to show \eqref{e:vdansuncycle}.

In the sequel, a graph is seen as a quadruple $G=(V,\vec{E},\iota,o)$ where $V$ and $\vec{E}$ are countable sets (respectively the set of vertices and half-edges), $o:\vec{E}\rightarrow V$ is a map and $\iota:\vec{E}\rightarrow \vec{E}$ is a map satisfying $\iota^2(e)=e$ and $\iota(e)\neq e$ for any $e\in\vec{E}$. Thus, $\iota$ defines an equivalence classes on $\vec{E}$, $e\sim f$ if and only if $e=\iota(f)$ with two elements in each equivalence class. An equivalence class is called an edge, the edge set is denoted by $E$.  We interpret $o(e)$ as the origin vertex of the directed edge $e$ and $t(e)=o(\iota(e))$ as the end vertex of $e$.

Let us reformulate the defintion of $n$-lifts provided in Definition \ref{d:randomnlift}. For an integer $n\geq 1$, let $S_n(G)$ be the family of permutations $(\sigma_e)_{e\in\vec{E}}$ such that $\sigma_{\iota(e)}=\sigma_e^{-1}$. A $n$-lift of $G$ is a graph $G_n=(V_n,\vec{E}_n,\iota_n,o_n)$ such that 
$$
V_n=V\times[n], \qquad \vec{E}_n=\vec{E}\times[n]
$$
and, for some $\sigma\in S_n(G)$, for any $(e,i)\in \vec{E}_n$
$$
\iota_n(e,i)=(\iota(e),\sigma_e(i)) \quad \text{and} \quad o_n(e,i)=(o(e),i).
$$

For $\mathbf{v}=(v,i)\in V_n$, we set 
$$
\vec{E}_n(v)=\{\mathbf{e}\in \vec{E}_n: o_n(\mathbf{e})=\mathbf{v}\}=\{(e,i):o(e)=v\}.
$$
We fix $\mathbf{v}\in V_n$ and we explore its neighborhood step by step. We start with $A_0=\vec{E}_n(\mathbf{v})$. At stage $t\geq 0$, if $A_t$ is not empty, we pick $\mathbf{e}_{t+1}=(e_{t+1},i_{t+1})$ in $A_t$ with $o_n(\mathbf{e}_{t+1})$ at minimal graph distance from $\mathbf{v}$ (we break ties with lexicographic order). We set $\mathbf{f}_{t+1}=\iota_n(\mathbf{e}_{t+1})=(\iota(e_{t+1}),\sigma_{e_{t+1}}(i_{t+1}))$. If $\mathbf{f}_{t+1}\in A_t$, we set $A_{t+1}=A_t\setminus\{ \mathbf{e}_{t+1},\mathbf{f}_{t+1}\}$, and, otherwise, 
$$
A_{t+1}=\left(A_t\cup \vec{E}_n(o_n(\mathbf{f}_{t+1}))\right) \setminus \{\mathbf{e}_{t+1},\mathbf{f}_{t+1}\}.
$$
At stage $\tau\leq n |\vec{E}|$, $A_\tau$ is empty, and we have explored the connected component of $\mathbf{v}$. Before stage 
$$
T=\sum_{s=1}^{h-1}D(D-1)^{s-1}=O\left((D-1)^h\right),
$$
we have discovered the subgraph spanned by the vertices at distance at most $h$ from $\mathbf{v}$. Also, if $\mathbf{v}$ has a cycle in its $h$-neighborhood, then $S(\mathbf{v})=S_{\tau \wedge T} \geq 1$ where 
$$
S_t=\sum_{s=1}^t \varepsilon_s \quad \text{and}\quad \varepsilon_t = \mathbf{1}(\mathbf{f}_t\in A_{t-1})
$$
 for $t\geq 1$. At stage $t\geq 0$, for any $e\in\vec{E}$, at most $t$ values of $\sigma_e$ have been revealed and $|A_t|\leq D+(D-1)(t-1)$.

Let $\mathcal{F}_t$ be the $\sigma$-algebra generated by $(A_0,\ldots,A_t)$ and $\mathbb{P}_{\mathcal{F}_t}$ be its conditional probability distribution. Then, $\tau$ is a stopping time. Also, if $t<\tau\wedge T$, let $B_t=\{(\iota(e_{t+1}),i)\in A_t: i\in[n]\}$ and $n_t\leq t$ be the number of $s\leq t$ such that $\mathbf{f}_s$ or $\mathbf{e}_s$ is of the form $(\iota(e_{t+1}),i)$, $i\in[n]$. We find
$$
\mathbb{P}_{\mathcal{F}_t}(\varepsilon_{t+1}=1)=\frac{|B_t|}{n-n_t}\leq \frac{DT}{n}=q.
$$
Hence, from the union bound, taking $h=\lfloor c\log n\rfloor$, we obtain
$$
\mathbb{Q}_n(A_{\mathbf{v}}^{(G_n)})\leq \mathbb{P}(S(\mathbf{v})\geq 1)\leq qT=O\left(\frac{(D-1)^{2h}}{n}\right).
$$
which concludes the proof of \eqref{e:vdansuncycle}.


\end{document}